\documentclass[11pt]{article}

\usepackage[a4paper,left=2.5cm,right=2.5cm,top=2cm,bottom=2.5cm]{geometry}
\usepackage[utf8]{inputenc}
\usepackage[T1]{fontenc}
\usepackage[english]{babel}
\usepackage{amsthm}
\usepackage{amsmath}
\usepackage{amssymb}
\usepackage{mathrsfs}
\usepackage{setspace}
\usepackage{float}
\usepackage{multirow}
\usepackage{stmaryrd}
\usepackage{color}
\usepackage{wrapfig}
\usepackage{pifont}
\usepackage{array}
\usepackage[colorlinks=true,linkcolor=ocre,citecolor=ocre]{hyperref}
\usepackage{dsfont}
\usepackage{upgreek}
\usepackage{nicefrac}
\usepackage{tikz}{}
\usepackage{gensymb}
\usepackage{FiraSans}

\usepackage{graphicx}
\usepackage{caption}
\renewenvironment{proof}{\noindent{\sffamily{\textbf{Proof :}}}}{\begin{flushright}$\square$\end{flushright}}

\newcommand{\IE}{\mathbb{E}}
\newcommand{\IN}{\mathbb{N}}

\newcommand{\IR}{\mathbb{R}}

\newcommand{\IT}{\mathbb{T}}

\newcommand{\drm}{\mathrm d}

\newcommand{\CS}{\mathcal S}

\newcommand{\CC}{\mathcal C}
\newcommand{\CH}{\mathcal H}

\newcommand{\CB}{\mathcal B}

\newcommand{\SF}{\mathscr{F}}

\newcommand{\eps}{\varepsilon}

\newcommand{\Wick}[1]{\mathbf{:}#1\mathbf{:}}

\newcommand\N{\mathbb{N}}
\newcommand\T{\mathbb{T}}

\newcommand\R{\mathbb{R}}
\newcommand\C{\mathbb{C}}

\newcommand{\dd}{\mathrm{d}}
\newcommand{\enstq}[2]{\left\{#1~\middle|~#2\right\}}

\renewcommand{\Re}{\operatorname{Re}}
\renewcommand{\Im}{\operatorname{Im}}

\newcommand\loc{\mathrm{loc}}

\newcommand\supp{\mathrm{supp}}

\definecolor{ocre}{RGB}{64,123,121}
\definecolor{S}{rgb}{0.0,0.5,0.0}

\newcounter{item}
\numberwithin{item}{section}

\newtheorem{theorem}[item]{\sffamily Theorem}
\newtheorem{definition}[item]{\sffamily Definition}
\newtheorem{proposition}[item]{\sffamily Proposition}
\newtheorem{lemma}[item]{\sffamily Lemma}

\newtheorem*{theorem*}{\sffamily Theorem}
\newtheorem*{definition*}{\sffamily Definition}
\newtheorem*{proposition*}{\sffamily Proposition}
\newtheorem*{lemma*}{\sffamily Lemma}
\newtheorem*{corollary*}{\sffamily Corollary}

\usepackage[explicit]{titlesec}
\titleformat{\section}{\centering\Large\bfseries}{\thesection \ --}{0.7em}{\Large\bfseries #1}
\titleformat{\subsection}{\centering\large\bfseries}{\thesubsection \ --}{0.4em}{\large\bfseries #1}
\titleformat{\subsubsection}{\centering\bfseries}{\thesubsubsection \ --}{0.4em}{\bfseries #1}

\let\emph\relax
\DeclareTextFontCommand{\emph}{\bfseries\em}

\providecommand{\keywords}[1]
{
	{\footnotesize	
	\textbf{Keywords --} #1}
}

\title{\bfseries The logarithmic Schrödinger equation with spatial white noise on the full space}
\author{Quentin CHAULEUR and Antoine MOUZARD}
\date{}

\begin{document}

\maketitle
\abstract{We solve the Schrödinger equation with logarithmic nonlinearity and multiplicative spatial white noise on $\IR^d$ with $d\le 2$. Because of the nonlinearity, the regularity structures and the paracontrolled calculus can not be used. To solve the equation, we rely on an exponential transform that has proven useful in the context of other singular SPDEs.}
\vspace{0.5cm}



\section{Introduction}

We consider the logarithmic Schrödinger equation with spatial white noise
\begin{equation} \tag{SlogNLS}\label{SlogNLS}
i \partial_t u = \Delta u + u \xi + \lambda u \log |u|^2
\end{equation}  
on $\R^d$ with initial data $u_0$ and $\xi$ the spatial white noise. We work in dimension $d\le2$, both in the focusing case $\lambda >0$ and the defocusing case $\lambda<0$. Since the introduction of regularity structures \cite{Hai14} and paracontrolled calculus \cite{GIP}, a large class of singular stochastic PDEs have been studied. While the first theories originally dealt with parabolic PDEs, these methods were adapted in order to solve dispersive singular PDEs such as the polynomial nonlinear Schrödinger equation with \cite{GUZ,Mouzard22,MouzardZachhuber22,Ugurcan22}. This approach relies on a construction of the Anderson Hamiltonian
\begin{equation*}
H=\Delta+\xi
\end{equation*}
as a self-adjoint operator which allows the resolution of linear and nonlinear associated evolution PDEs such as the Schrödinger equation. In particular, \cite{MouzardZachhuber22} obtained Strichartz inequalities on compact surfaces which allow the resolution of the cubic NLS equation with initial data in the energy space of $H$. Another approach was used by Debussche and Weber \cite{debussche2018} on $\T^2$, see also \cite{debussche2019,tzvetkov2022,TzvetkovVisciglia22,debussche2023} for more general results on $\T^2$ and $\R^2$. They consider an exponential transform first used for the Parabolic Anderson Model (PAM) equation on $\R^2$ by Hairer and Labbé \cite{hairer2015} with the new variable $v=e^Xu$ where $X$ is a random field solution to
\begin{equation*}
\Delta X=\xi.
\end{equation*}
In particular, this allows the resolution of different singular stochastic PDEs without regularity structures and paracontrolled calculus. This transformation was also used for example for constructive Quantum Field Theory \cite{JagannathPerkowski23,BDFT23one} or the Anderson form \cite{MatsudaZuijlen22} with the additional use of regularity structures. While the resolution of parabolic singular SPDEs relies on the regularizing properties of the heat semigroup, one uses the Hamiltonian structure of the equation and its conservative quantities in the dispersive case. In particular, this requires well-prepared initial data depending on the noise.

\medskip

In this work, we solve the logarithmic Schrödinger equation on $\R^2$ with spatial white noise using the exponential transform. In particular, the nonlinearity is not Lipschitz which prevents from using any generalized Taylor expansion or paracontrolled expansion. Note that in the context of nonlinear Schrödinger equations, the logarithmic nonlinearity can be seen as the formal limit $p \rightarrow 1$ of NLS equation with polynomial nonlinearities of the form $\lambda |u|^{p-1}u$, see the recent survey \cite[Section 7]{carles2022survey} for a better understanding of this point of view. From this perspective, this work can be considered as a natural continuation of the articles \cite{debussche2019,MouzardZachhuber22,tzvetkov2022,TzvetkovVisciglia22,debussche2023} where the authors successively increased the range of admissible nonlinearity powers $p>1$, using more and more involved tools from dispersive theory such as Strichartz estimates or modified energies techniques. In the deterministic logarithmic setting \cite{carles2022survey,carles2018}, Strichartz inequalities or even more involved dispersive properties has not yet shown to be useful, and one has to rely only on the Hamiltonian structure of the equation alongside particular algebraic properties of the nonlinearity. Hopefully, the exponential transform from \cite{hairer2015} interacts well with the logarithmic nonlinearity of \eqref{SlogNLS}. In fact, with the new variable $v=e^Xu$, the equation rewrites formally as
\begin{equation*}
i \partial_t v =\Delta v - 2\nabla X\cdot \nabla v+|\nabla X|^2v-2\lambda Xv+\lambda v\log|v|^2
\end{equation*}
which is better behaved since the roughest term $\xi$ is canceled. In two dimensions, the Hölder regularity of the noise is $-1-\kappa$ for any $\kappa>0$ hence $\nabla X$ is a distribution and we have to replace $|\nabla X|^2$ by a Wick product $\Wick{|\nabla X|^2}$ as for the polynomial equation, this is the renormalization procedure. Then one can almost surely solve this equation for initial data $v_0=e^Xu_0\in\CH^2$. The strategy is to consider a regularization $\xi_\eps$ of the noise and the equation
\begin{equation*}
i \partial_t v_\eps =\Delta v_\eps - 2\nabla X_\eps\cdot \nabla v_\eps+\Wick{|\nabla X_\eps|^2}v_\eps-2\lambda X_\eps v_\eps+\lambda v_\eps\log|v_\eps|^2
\end{equation*}
with $\Delta X_\eps=\xi_\eps$. A suitable $\CH^2$ bound for this solution has eventually to be obtained in order to recover a solution in the limit $\eps$ to $0$. In unbounded space, there is an additional difficulty due to the fact that the noise does not decrease at infinity. In fact, the noise $\xi$ has a sub-polynomial growth at infinity and only belongs to weighted Hölder spaces. A solution is to suppose some decrease for the initial data $v_0\in\CH_{\mu_0}^2$ and to propagate this to the solution, as done in \cite{debussche2019,debussche2023}. Note that in the case of the logarithmic Schrödinger equation, working with weighted initial data is natural as it already appears in the deterministic equation, see the work of Carles and Gallagher \cite{carles2018}. To the best of our knowledge, this is the first pathwise resolution of a singular SPDE with a nonlinearity that is not locally Lipschitz, see for example \cite{PerkowskiRosati21} where Perkowski and Rosati construct martingale solutions to an equation of this type.

\medskip

For the regularized equation, one still has to deal with the logarithmic nonlinearity and the growth of the potential on $\R^2$. Indeed, the noise $\xi$ belongs to $\CC_{\loc}^{-1-\kappa}$ for any $\kappa>0$ however it does not decrease at infinity since its law is invariant by translation. The regularized noise $\xi_\eps=\xi*\rho_\eps$ is a centered Gaussian field with covariance
\begin{equation*}
\IE\big[\xi_\eps(x)\xi_\eps(y)\big]=\int_{\R^2}\rho_\eps(x-z)\rho_\eps(y-z)\drm z
\end{equation*}
and one can prove that $\xi_\eps$ converges to $\xi$ in a weighted Hölder space. The idea is to introduce a new parameter $\delta>0$ which acts as a regularization of the nonlinearity with $v_{\eps,\delta}$ the solution to
\begin{equation*}
i \partial_t v_{\eps,\delta} =\Delta v_{\eps,\delta} - 2\nabla X_\eps\cdot \nabla v_{\eps,\delta}+\Wick{|\nabla X_\eps|^2}v_{\eps,\delta}-2\lambda X_\eps v_{\eps,\delta}+\lambda v_{\eps,\delta}\log\big(\delta+|v_{\eps,\delta}|^2\big)
\end{equation*}
which has a unique global solution for $v_0\in\CH^2$, see for example \cite{cazenave}. Adapting the method from Carles and Gallagher \cite{carles2018}, we are able to recover a unique solution in the limit $\delta$ to $0$.

\medskip

In this work, we solve \eqref{SlogNLS} on $\R^d$ with $d\le2$. In one dimension, this stochastic PDE is not singular and one can obtain a unique global solution for initial data $u_0\in L_{\mu_0}^2\cap\CH^1$. A natural question is the propagation of the regularity, in general in terms of Sobolev spaces. In the presence of white noise, the irregularity of the noise prevents the propagation of classical regularity, however one can still study the propagation of Sobolev spaces associated with the Anderson Hamiltonian. We prove using the exponential transform that one can almost propagate $e^{-X}\CH^2$ with $X$ the random field solution to $\Delta X=\xi$. In two dimensions, this is a singular stochastic PDE and we use the exponential transform to get global well-posedness for initial data $u_0\in e^{-X}\CH_{\mu_0}^2$. Our work applies to the same equation on $\T^d$ as no dispersive effects are used throughout the proofs, we focus on the unbounded case which is harder since one has to deal with weighted functional spaces.

\medskip

In Section \ref{SectionFunctionalSpaces}, we recall the needed tools from harmonic analysis with the Paley-Littlewood decomposition, weighted Besov spaces, product rule and duality. In Section \ref{SectionStoc}, we give the needed stochastic bounds on the random fields and the renormalization of the equation with the Wick product. In Section \ref{SectionDeter}, we solve the equation for a deterministic regular potential with a possible growth at infinity. In Section \ref{Section1D}, we solve \eqref{SlogNLS} on $\R$ and prove the propagation of regularity for initial data in the space $e^{-X}\CH^2$. Finally, we solve \eqref{SlogNLS} on $\R^2$ using the exponential transform in Section~\ref{Section2D}.

\section{Functional spaces} \label{SectionFunctionalSpaces}

Since the law of the white noise is invariant by translation, it does not decay at infinity. To deal with this, we work in weighted Lebesgue and Besov spaces. For any $\mu\in\R$ and $p\in[1,\infty]$, we consider
\begin{equation*}
\|u\|_{L_\mu^p(\R^d)}=\Big(\int_{\R^d}\langle x\rangle^\mu|f(x)|^p\drm x\Big)^{\frac{1}{p}}
\end{equation*}
with $\langle x\rangle=\sqrt{1+|x|^2}$. An important tool is the Paley-Littlewood decomposition
\begin{equation*}
u=\sum_{n\ge0}\Delta_nu
\end{equation*}
where 
\begin{equation*}
\big(\Delta_nu\big)(x):=2^{d(n-1)}\int_{\R^d}\chi\big(2^{n-1}(x-y)\big)u(y)\drm y
\end{equation*} 
with $\chi\in\CS(\R^d)$ and $\supp\ \widehat\chi\subset\{\frac{1}{2}\le|z|\le 2\}$ for $n\ge1$ and
\begin{equation*}
\big(\Delta_0u\big)(x):=\int_{\R^d}\chi_0(x-y)u(y)\drm y
\end{equation*}
with $\chi_0\in\CS(\R^d)$ and $\supp\ \widehat\chi_0\subset\{|z|\le 1\}$. Most of the following definitions and properties can be found in \cite[Section 4]{edmunds1996}. See also \cite{debussche2019} and references therein, in particular the book \cite{BCD}.

\medskip

\begin{definition}
Let $\mu,\alpha\in\R$ and $p,q\in[1,\infty)$. The weighted Besov space $\CB_{p,q,\mu}^\alpha$ is the set of distribution $u\in\CS'(\R^d)$ such that
\begin{equation*}
\|u\|_{\CB_{p,q,\mu}^\alpha}:=\Big(\sum_{n\ge0}2^{\alpha n q}\|\Delta_nu\|_{L_\mu^p(\R^d)}^q\Big)^{\frac{1}{q}}<\infty.
\end{equation*}
\end{definition}

\medskip

For $p=q=2$, one recovers the usual weighted Sobolev spaces $\CH_\mu^\alpha=\CB_{2,2,\mu}^\alpha$ with
\begin{equation*}
\|u\|_{\CH_\mu^\alpha(\R^d)}=\|(\SF^{-1}\langle\cdot\rangle^\alpha\SF)u\|_{L_\mu^2(\R^d)}.
\end{equation*}
We also denote the case $p=q=\infty$ as $\CC_\mu^\alpha=\CB_{\infty,\infty,\mu}^\alpha$ which corresponds to the usual weighted Hölder spaces for $\alpha\in\R_+\backslash\N$. Moreover, there exist constants $C_1,C_2>0$ depending on the spaces parameters such that
\begin{equation*}
C_1\|\langle\cdot\rangle^\mu u\|_{\CB_{p,q,0}^\alpha}\le\|u\|_{\CB_{p,q,\mu}^\alpha}\le C_2\|\langle\cdot\rangle^\mu u\|_{\CB_{p,q,0}^\alpha},
\end{equation*}
and $\CB_{p,q,0}^\alpha$ corresponds to the usual Besov spaces hence weighted Besov spaces satisfy the following embeddings.

\medskip

\begin{lemma} \label{besov_embeddings}
Let $p_1,p_2,q_1,q_2\in[1,\infty]$ and $\mu,\mu'\in\R$ such that $p_1\le p_2,q_1\le q_2$ and $\mu_1\ge\mu_2$. For all $\alpha\in \R$, one has the Besov embeddings
\begin{equation*}
\mathcal{B}^\alpha_{p_1,q_1,\mu_1}(\R^d) \subset  \mathcal{B}^{\alpha-d\left(\frac{1}{p_1}-\frac{1}{p_2} \right)}_{p_2,q_2,\mu_2}(\R^d).
\end{equation*}
as well as the Sobolev embeddings
\begin{equation*}
\forall\alpha\ge \frac{d}{2}-\frac{d}{p},\quad \CH^\alpha_{\mu}(\R^d) \subset  L^p_{\mu}(\R^d)
\end{equation*}
for $p\in[2,\infty]$. Finally, the embedding
\begin{equation*}
\CH_{\mu_1}^{\alpha_1}(\R^d)\hookrightarrow\CH_{\mu_2}^{\alpha_2}(\R^d)
\end{equation*}
is compact for $\alpha_1>\alpha_2$ and $\mu_1>\mu_2$.
\end{lemma}

\medskip

The following lemma is a useful interpolation estimate that we shall use in the following, and which can be found in \cite[Theorem 3.8]{sickel2014}.

\medskip

\begin{lemma} \label{interpolation_estimate_Besov} 
Let $p_0,p_1,q_0,q_1 \in[1,\infty]$ and $\alpha_0,\alpha_1,\mu_0,\mu_1 \in \R$ and $p,q,\alpha,\mu$ such that
\begin{equation*} 
\frac{1}{p}=\frac{1-\theta}{p_0}+\frac{\theta}{p_1},\quad \frac{1}{q}=\frac{1-\theta}{q_0}+\frac{\theta}{q_1},
\end{equation*}
and
\begin{equation*}
\alpha=(1-\theta)\alpha_0+\theta\alpha_1, \quad  \mu=(1-\theta)\mu_0 + \theta \mu_1
\end{equation*}
for $\theta \in[0,1]$. Then there exists $C>0$ such that
\begin{equation*} 
\|u\|_{\mathcal{B}^\alpha_{p,q,\mu}(\R^d)} \leq C \|u\|_{\mathcal{B}^{\alpha_0}_{p_0,q_0,\mu_0}(\R^d)}^{1-\theta} \|u\|_{\mathcal{B}^{\alpha_1}_{p_1,q_1,\mu_1}(\R^d)}^{\theta}.
\end{equation*}
In particular, we have
\begin{equation*}
\|u\|_{\CH_\mu^\alpha}\le C\|u\|_{\CH_{\mu_0}^{\alpha_0}}^{1-\theta}\|u\|_{\CH_{\mu_1}^{\alpha_1}}^\theta.
\end{equation*}
\end{lemma}

\medskip

In general, one can only multiply a distribution by a test function. The following lemma gives a product rule in weighted Besov space, as a generalisation of Young condition, see~\cite[Section 2.8.2]{triebel1983}.

\medskip

\begin{lemma}\label{product_estimate_Besov} 
Let $\alpha_1,\alpha_2 \in \R$ such that $\alpha_1+\alpha_2>0$. Let $\mu_1,\mu_2 \in \R$ with $\mu=\mu_1+\mu_2$ and $p_1,p_2\in[1,\infty]$ with $\frac{1}{p}=\frac{1}{p_1}+\frac{1}{p_2}$. Then for any $\kappa >0$, there exists a constant $C>0$ such that
\begin{equation*} 
\|uv\|_{\mathcal{B}_{p,p,\mu}^{\alpha-\kappa}(\R^d)} \leq C \| u \|_{\mathcal{B}_{p_1,p_1,\mu_1}^{\alpha_1}(\R^d)} \| v \|_{\mathcal{B}_{p_2,p_2,\mu_2}^{\alpha_2}(\R^d)}
\end{equation*}
where $\alpha=\min(\alpha_1,\alpha_2)$.
\end{lemma}

\medskip

As we will control energy associated to dispersive PDEs, the following duality result from in weighted Besov spaces from \cite[Theorem 2.11.2]{triebel1983} will be useful.

\medskip

\begin{lemma}\label{duality_estimate_Besov} 
Let $\alpha,\mu \in \R$ and $p,q \in [1,\infty ]$. Then there exists a constant $C>0$ such that
\begin{equation*} 
\Big| \int_{\R^d} u(x)v(x)\dd x \Big| \leq C \|u\|_{\mathcal{B}_{p,q,\mu}^\alpha(\R^d)} \|v\|_{\mathcal{B}_{p',q',-\mu}^{-\alpha}(\R^d)},
\end{equation*}
where $\frac{1}{p}+\frac{1}{p'}=1$ and  $\frac{1}{q}+\frac{1}{q'}=1$ for any $u,v\in\CS(\IR^d)$.
\end{lemma}

\medskip

In order to solve the logarithmic Schrödinger equation with an unbounded potential, we will also need the following inequality in weighted Lebesgue spaces.

\medskip

\begin{lemma} \label{logarithmic_lemma}
Let $m \in \N^*$, $\eta>0$ and $ \mu,\mu_0 \in \R$ such that $0 \leq \mu <\mu_0$.
If $\eta<\frac{2(\mu_0-\mu)}{\frac{d}{2}+\mu_0}$, then there exists a constant $C>0$ such that
\begin{equation*}
\int_{\R^d} \langle x \rangle^{2\mu} |u(x)|^2 \left| \log |u(x)|^2 \right|^m \dd x \le C \|u\|_{L^2_{\mu_0}}^{\frac{d\eta}{2\mu_0}+\frac{2\mu}{\mu_0}} \|u\|_{L^2}^{2-\eta-\frac{d\eta}{2\mu_0}-\frac{2\mu}{\mu_0}} + \|u\|^{2+\eta}_{L^{2+\eta}_{\frac{2\mu}{2+\eta}}}.
\end{equation*}
Moreover, there exists a constant $C'>0$ such that
\begin{equation*}
\|u\|^{2+\eta}_{L^{2+\eta}_{\frac{2\mu}{2+\eta}}} \le C'\min\Big( \| u \|_{H^1_{\frac{4\mu}{d\eta}}}^{\frac{d\eta}{2}} \| u \|_{L^2}^{2+\eta-\frac{d\eta}{2}}, \| u \|_{H^1}^{\frac{d\eta}{2}} \| u \|_{L^2_{\frac{2\mu}{2+\eta-\frac{d \eta}{2}}}}^{2+\eta-\frac{d\eta}{2}} \Big).
\end{equation*}
In particular, this gives
\begin{equation*}
\int_{\R^d}|u(x)|^2 \left| \log |u(x)|^2 \right|^m \dd x \le C \|u\|_{L^2_{\mu_0}}^{\frac{d\eta}{2\mu_0}} \|u\|_{L^2}^{2-\eta-\frac{d\eta}{2\mu_0}} +C \| u \|_{\CH^1}^{\frac{d\eta}{2}} \| u \|_{L^2}^{2+\eta-\frac{d\eta}{2}}.
\end{equation*}
\end{lemma}

\begin{proof}
For any $m\in\IN^*$ and $\eta>0$, there exists a constant $C>0$ such that
\begin{equation*}
|\log y|^m\le Cy^\eta+Cy^{-\eta}
\end{equation*}
for $y>0$ thus
\begin{equation*}
\int_{\R^d}\langle x \rangle^{2\mu} |u(x)|^2 \left| \log |u(x)|^2 \right|^m \dd x \le C\int_{\R^d}\langle x \rangle^{2\mu} |u(x)|^{2-\eta} \dd x + C\int_{\R^d}\langle x \rangle^{2\mu} |u(x)|^{2+\eta} \dd x.
\end{equation*}
For the first term, let $R>0$ and write
\begin{equation*}
\int_{\R^d} \langle x \rangle^{2\mu} |u(x)|^{2-\eta} \dd x = \int_{|x| \leq R} \langle x \rangle^{2\mu} |u(x)|^{2-\eta} \dd x + \int_{|x|>R} \langle x \rangle^{2\mu} |u(x)|^{2-\eta} \dd x
\end{equation*}
following for example \cite[Lemma 6.2]{carles2004}. For $p=\frac{2}{2-\eta}\in(1,\infty)$ and $p'=\frac{2}{\eta}$ such that $\frac{1}{p}+\frac{1}{p'}=1$, Hölder inequality gives
\begin{align*}
\int_{\IR^d}\langle x\rangle^{2\mu}|u(x)|^{2-\eta}\drm x&\le\Big(\int_{|x|\le R}\langle x\rangle^{2\mu p'}\drm x\Big)^{\frac{1}{p'}}\Big(\int_{|x|\le R}|u(x)|^{(2-\eta)p}\drm x\Big)^{\frac{1}{p}}\\
&\quad+\Big(\int_{|x|>R}\langle x\rangle^{(2\mu-\alpha)p'}\drm x\Big)^{\frac{1}{p'}}\Big(\int_{|x|>R}\langle x\rangle^{\alpha p}|u(x)|^{(2-\eta)p}\drm x\Big)^{\frac{1}{p}}\\
&\le\Big(\int_{|x|\le R}\langle x\rangle^{2\mu p'}\drm x\Big)^{\frac{\eta}{2}}\Big(\int_{|x|\le R}|u(x)|^2\drm x\Big)^{\frac{2-\eta}{2}}\\
&\quad+\Big(\int_{|x|>R}\langle x\rangle^{(2\mu-\alpha)p'}\drm x\Big)^{\frac{\eta}{2}}\Big(\int_{|x|>R}\langle x\rangle^{\alpha p}|u(x)|^2\drm x\Big)^{\frac{2-\eta}{2}}
\end{align*}
with $\alpha>0$ such that $(2\mu-\alpha)p'<-d$. This gives the condition
\begin{equation*}
\alpha>\frac{\eta d}{2}+2\mu
\end{equation*}
hence we can take $\alpha=\frac{2\mu_0}{p}=(2-\eta)\mu_0$ for $\eta$ small enough to get
\begin{align*}
\int_{\IR^d}\langle x\rangle^{2\mu}|u(x)|^{2-\eta}\drm x&\le C'R^{2\mu+\frac{d\eta}{2}}\Big(\int_{\IR^d}|u(x)|^2\drm x\Big)^{\frac{2-\eta}{2}}\\
&\quad	+C'R^{2\mu-(2-\eta)\mu_0+\frac{d\eta}{2}}\Big(\int_{\IR^d}\langle x\rangle^{2\mu_0}|u(x)|^2\drm x\Big)^{\frac{2-\eta}{2}}
\end{align*}
for a constant $C'=C'(\eta,\mu,d)>0$. Optimizing this bound in $R>0$ gives
\begin{equation*}
R=\left(\frac{\|u\|_{L_{\mu_0}^2}}{\|u\|_{L^2}}\right)^{\frac{1}{\mu_0}}
\end{equation*}
and we get
\begin{equation*}
\int_{\IR^d}\langle x\rangle^{2\mu}|u(x)|^{2-\eta}\drm x\le2C'\|u\|_{L_{\mu_0}^2}^{\frac{2\mu}{\mu_0}+\frac{d\eta}{2\mu_0}}\|u\|_{L^2}^{2-\eta-\frac{2\mu}{\mu_0}-\frac{d\eta}{2\mu_0}},
\end{equation*}
which concludes that
\begin{equation*}
\int_{\R^d}\langle x \rangle^{2\mu} |u(x)|^2 \left| \log |u(x)|^2 \right|^m \dd x \le 2CC'\|u\|_{L_{\mu_0}^2}^{\frac{2\mu}{\mu_0}+\frac{d\eta}{2\mu_0}}\|u\|_{L^2}^{2-\eta-\frac{2\mu}{\mu_0}-\frac{d\eta}{2\mu_0}}+ C\|u\|_{L_{\frac{2\mu}{2+\eta}}^{2+\eta}}^{2+\eta}.
\end{equation*}
The condition $\eta$ small enough is explicit and given by $(2-\eta)\mu_0>\frac{\eta d}{2}+2\mu$, that is
\begin{equation*}
\eta<\frac{2(\mu_0-\mu)}{\frac{d}{2}+\mu_0}.
\end{equation*}
Lemma \ref{besov_embeddings} gives
\begin{equation*}
\|u\|_{L_{\frac{2\mu}{2+\eta}}^{2+\eta}}\le C\|u\|_{\CH_{\frac{2\mu}{2+\eta}}^{\frac{d}{2}-\frac{d}{2+\eta}}}
\end{equation*}
with $C=C(\eta,\mu)>0$ and Lemma \ref{interpolation_estimate_Besov} gives
\begin{equation*}
\|u\|_{\CH_{\frac{2\mu}{2+\eta}}^{\frac{\eta d}{2(2+\eta)}}}\le C'\|u\|_{L^2}^{1-\theta}\|u\|_{\CH_{\mu_\theta}^{\alpha_\theta}}^\theta
\end{equation*}
for $\theta\in[0,1]$ with $C'=C'(\eta,\theta,\mu)>0$ and
\begin{equation*}
\alpha_\theta=\frac{\eta d}{2\theta(2+\eta)}\quad\text{and}\quad\mu_\theta=\frac{2\mu}{\theta(2+\eta)}.
\end{equation*}
Hence taking $\theta=\frac{d\eta}{2(2+\eta)}$ yields
\begin{equation*}
\|u\|_{L_{\frac{2\mu}{2+\eta}}^{2+\eta}}^{2+\eta}\le C''\|u\|_{L^2}^{2+\eta-\frac{d\eta}{2}}\|u\|_{\CH_{\frac{4\mu}{d\eta}}^1}^{\frac{d\eta}{2}}.
\end{equation*}
Following the same path, one also gets
\begin{equation*}
\|u\|^{2+\eta}_{L^{2+\eta}_{\frac{2\mu}{2+\eta}}} \le C''\| u \|_{\CH^1}^{\frac{d\eta}{2}} \| u \|_{L^2_{\frac{2\mu}{2+\eta-\frac{d \eta}{2}}}}^{2+\eta-\frac{d\eta}{2}}
\end{equation*}
which completes the proof.
\end{proof}

\section{Stochastic bounds and renormalization} \label{SectionStoc}

In this section, we give the bounds on random fields needed to solve \eqref{SlogNLS} on $\R$ and~$\R^2$. In particular, we perform the renormalization probabilistic step with the definition of the Wick square $\Wick{|\nabla X|^2}$ on $\R^2$. The noise $\xi$ is the Gaussian random distribution such that $\IE\big[\langle\xi,\varphi\rangle\big]=0$ and
\begin{equation*}
\IE\big[\langle\xi,\varphi\rangle\langle\xi,\psi\rangle\big]=\langle\varphi,\psi\rangle_{L^2(\IR^d)}
\end{equation*}
for any $\varphi,\psi\in\CS(\IR^d)$. The covariance can also be written formally as
\begin{equation*}
\IE\big[\xi(x)\xi(y)\big]=\delta_0(x-y).
\end{equation*}
The following proposition gives the regularity of the noise. We give its proof to give a flavor of the arguments but refer to \cite{debussche2019} for the proofs of all the finer stochastic bounds needed throughout the analysis.

\medskip

\begin{proposition} \label{estimate_white_noise}
For any $\alpha<-\frac{d}{2}$ and $\mu>0$, we have
\begin{equation*}
\xi\in\CC_{-\mu}^{\alpha}(\R^d)
\end{equation*}
almost surely.
\end{proposition}

\medskip

\begin{proof}
Since the noise is Gaussian, we have
\begin{equation*}
\IE\big[\langle \xi,\varphi\rangle^p\big]\le(p-1)^{\frac{p}{2}}\IE\big[\langle\xi,\varphi\rangle^2\big]^{\frac{p}{2}}
\end{equation*}
for any test function $\varphi$, this is usually refered to as Gaussian hypercontractivity. In order to use this, we estimate the Besov norm $\CB_{p,p,-\mu}^\gamma$ for $p$ large and use the embedding
\begin{equation*}
\CB_{p,p,-\mu}^\gamma(\R^d) \hookrightarrow  \CB^{\gamma-\frac{d}{p}}_{\infty,\infty,-\mu}(\R^d).
\end{equation*}
Denoting $K_n=2^{d(n-1)}\chi(2^{n-1} \cdot)$ as in the definitions of the Besov spaces in Section \ref{SectionFunctionalSpaces}, we have
\begin{align*}
\IE\Big[\|\Delta_n\xi\|_{L_{-\mu}^p(\R^d)}^p\Big]&=\int_{\IR^d}\IE\big[\langle \xi,K_n(x-\cdot)\rangle^p\big]\langle x\rangle^{-p\mu}\drm x\\
&\le(p-1)^{\frac{p}{2}}\int_{\IR^d}\IE\big[\langle \xi,K_n(x-\cdot)\rangle^2\big]^{\frac{p}{2}}\langle x\rangle^{-p\mu}\drm x\\
&\le(p-1)^{\frac{p}{2}}\|K_n\|_{L^2(\IR^d)}^p\int_{\IR^d}\langle x\rangle^{-p\mu}\drm x
\end{align*}
using that $\xi$ is an isometry from $L^2(\IR^d)$ to $L^2(\Omega)$. This is finite as long as $p\mu>d$ and we have
\begin{equation*}
\|K_n\|_{L^2(\IR^d)}^2=2^{2d(n-1)}\|K(2^{n-1}\cdot)\|_{L^2(\IR^d)}^2=2^{d(n-1)}\|K\|_{L^2(\IR^d)}^2.
\end{equation*}
We get
\begin{equation*}
\IE\Big[\|\Delta_n\xi\|_{L_{-\mu}^p}^p\Big]\le(p-1)^{\frac{p}{2}}2^{p(n-1)\frac{d}{2}}\|K\|_{L^2(\IR^d)}^2\int_{\IR^d}\langle x\rangle^{-p\mu}\drm x
\end{equation*}
hence for any $\gamma<-\frac{d}{2}$, there exists a constant $C=C(\gamma)>0$ such that
\begin{equation*}
\IE\Big[2^{np\gamma}\|\Delta_n\xi\|_{L_{-\mu}^p}^p\Big]\le C2^{\frac{1}{2}pn(\gamma+\frac{d}{2})}.
\end{equation*}
For $p$ large enough, this is a convergent series thus
\begin{equation*}
\IE\big[\|\xi\|_{\CC_{-\mu}^{\alpha}}\big]<\infty
\end{equation*}
for any $\alpha<-\frac{d}{2}$ and $\mu>0$, which completes the proof since a random variable in $L^1(\Omega)$ is finite almost surely.
\end{proof}

Let $G\in C^{\infty}(\IR^2\backslash\{0\})$ such that $\supp\ G\subset B(0,1)$ and that $G$ coincide with the Green function of the Laplacian on a small ball around $0$. Then $X:=G*\xi$ is a solution to
\begin{equation*}
\Delta X=\xi+\varphi*\xi
\end{equation*}
with a suitable function $\varphi\in C_c^\infty(\IR^2)$. Consider a regularization of the noise $\xi_\eps=\xi*\rho_\eps$ with $\rho_\eps(\cdot)=\eps^{-2}\rho(\eps^{-1}\cdot)$ and $\rho$ a smooth positive function such that $\int_{\IR^2}\rho(x)\drm x=1$. Then $\xi_\eps$ converges to $\xi$ as $\eps$ goes to $0$ in $\CC_{-\mu}^{-\frac{d}{2}-\kappa}$ and one can consider $X_\eps:=G*\xi_\eps$ which converges to $X$ as $\eps$ goes to $0$ in $\CC^{2-\frac{d}{2}-\kappa}$. In one dimension, $\nabla X$ is a function hence the square $|\nabla X|^2$ is well-defined. In two dimensions, $\nabla X$ is only a distribution hence the square $|\nabla X|^2$ is ill-defined and the family of functions $|\nabla X_\eps|^2$ diverges as $\eps$ goes to $0$. This divergence is described by the Wick square as stated in the following proposition.

\medskip

\begin{proposition}
There exists a distribution $\Wick{|\nabla X|^2}\in\CC_{-\mu}^{-\kappa}(\IR^2)$ for any $\kappa>0$ such that
\begin{equation*}
\Wick{|\nabla X|^2}=\lim_{\eps\to0}\Big(|\nabla X_\eps|^2-\IE\big[|\nabla X_\eps|^2\big]\Big)
\end{equation*}
in $\CC_{-\mu}^{-\kappa}(\IR^2)$. Moreover, the mean $\IE\big[|\nabla X_\eps|^2\big]$ diverges as $\log(\eps)$.
\end{proposition}

\medskip

In order to solve our equation, we will need the following bounds proved in \cite{debussche2019}, see Lemmas 2.7, 2.8 and 2.10. In particular, not only $X$ has sub-polynomial growth but also $e^X$ and $e^{-X}$ which will be crucial to our results. Similar bounds can be obtained on the line without the renormalization procedure, we do not give the details since it is similar.

\medskip

\begin{lemma}\label{estimate_Y_R2}
For any $\mu>0$, $\alpha \in (0,1),\beta\in\R$ and $a\in\R$, we have
\begin{equation*}
\| X_{\eps}\|_{\mathcal{C}^{\alpha}_{-\mu}(\R^2)} + \|\Wick{\nabla X_{\eps}^2}\|_{\mathcal{C}^{\alpha-1}_{-\mu}(\R^2)} + \| e^{a X_{\eps}}\|_{\mathcal{C}^{\alpha}_{-\mu}(\R^2)} + \| \varphi \ast \xi_{\eps} \|_{\mathcal{C}^{\beta}_{-\mu}(\R^2)} \leq C
\end{equation*}
with $C>0$ a random constant bounded in $L^p$ for any $p\in[1,\infty)$.
\end{lemma}

\begin{lemma}\label{estimate_Lp_Y_R2}
For any $\mu \in (0,1)$ and $p \in (2/\mu,\infty)$, we have
\begin{equation*}
\| \nabla X_{\eps}\|_{L^p_{-\mu}(\R^2)}^2 + \|\Wick{\nabla X_{\eps}^2}\|_{L^p_{-\mu}(\R^2)}  \leq C|\log \eps|
\end{equation*}
with $C>0$ a random constant bounded in $L^p$ for any $p\in[1,\infty)$.
\end{lemma}

\begin{lemma}\label{estimate_diff_Y_R2}
For any $\mu>0,\alpha\in (0,1)$ and $\kappa \in (0,1-\alpha)$, we have
\begin{equation*}
\|X_{\eps}-X\|_{\mathcal{C}^\alpha_{-\mu}(\R^2)}^2 + \|\mathbf{:}\nabla X_{\eps}^2\mathbf{:}-\Wick{\nabla X^2}\|_{\mathcal{C}^{\alpha-1}_{-\mu}(\R^2)}  \leq C\eps^{\kappa}
\end{equation*}
with $C>0$ a random constant bounded in $L^p$ for any $p\in[1,\infty)$.
\end{lemma}

\section{The deterministic equation} \label{SectionDeter}

In this section, we consider the deterministic logarithmic Schrödinger equation 
\begin{equation} \tag{logNLS} \label{logNLS}
i \partial_t u =\Delta u + V u+ \lambda u \log |u|^2
\end{equation}
with initial data $u_0\in L_{\mu_0}^2(\IR^d)\cap\CH^1(\IR^d)$ and a potential $V \in \mathcal{C}^{\infty}(\R^d)$ which has sub-polynomial growth at infinity, namely
\begin{equation*}
V  \in L^{\infty}_{-\mu}(\R^d)
\end{equation*}
for all $\mu>0$ and $d\ge1$. Two main problems prevent the use of classical arguments in order to get global well-posedness, the fact that the logarithmic nonlinearity is not Lipschitz at the origin and the growth of the potential $V$ at infinity. In order to solve the equation, we work in weighted Besov spaces with a space truncation of $V$ and a regularization of the logarithmic nonlinearity. This gives a global solution that passes to the limit with respect to the two parameters which allows to recover a solution to \eqref{logNLS}.

\medskip

\begin{proposition} \label{existence_logNLS_R_prop}
Let $0<\mu_0 \leq \frac12$ and $u_0 \in L_{\mu_0}^2\cap\CH^1$. There exists a unique solution~$u\in \mathcal{C}(\R;L_{\mu_0}^2\cap\CH^1)$ to \eqref{logNLS} with initial data $u_0$. If moreover $u_0\in\CH^2$, then $u\in \mathcal{C}(\R;L_{\mu_0}^2\cap\CH^2)$.
\end{proposition}

\medskip

\begin{proof}
Let $\chi \in \mathcal{C}^{\infty}_0(\R^d)$ a smooth positive compactly supported function such that $\supp\ \chi\subset B(0,2)$ and $\chi\equiv1$ on $B(0,1)$. Denote 
\begin{equation*}
\chi_n(x):=\chi\big(n^{-1}x\big)\quad\text{and}\quad V_n:=\chi_n V
\end{equation*} 
for $n\ge1$. Let $\delta>0$ and consider the regularized logarithmic Schrödinger equation
\begin{equation*}
i \partial_t u_n^{\delta} = \Delta u_n^{\delta} + V_n u_n^{\delta} + \lambda u_n^{\delta} \log \big(\delta+|u_n^{\delta}|^2\big) 
\end{equation*}
with initial data $u_n^\delta(0)=u_0$. As this equation is $L^2$-subcritical for any dimension and contains only bounded and smooth terms, there exists a unique solution $u_n^{\delta} \in \mathcal{C}(\R,\CH^1)$ for $u_0 \in\CH^1$, see for example \cite[Theorem 3.4.1]{cazenave}. The two conserved quantities are the mass
\begin{equation*}
M(u_n^\delta)=\int_{\IR^d}|u_n^\delta(x)|^2\drm x
\end{equation*}
and the energy 
\begin{equation*}
E_{\delta,n}(u_n^{\delta})=\int_{\R^d} |\nabla u_n^{\delta}(x)|^2\drm x- \int_{\R^d} V_n(x)|u_n^{\delta}(x)|^2\drm x  - \lambda \int_{\R^d} (\delta+|u_n^{\delta}(x)|^2) \log\big(\delta+|u_n^{\delta}(x)|^2\big)\drm x .
\end{equation*}
We first obtain an uniform bound in $\CH^1$ with respect to $\delta$. Differentiating the equation gives
\begin{equation*}
i \partial_t \nabla u_n^{\delta}  = \Delta \nabla u_n^{\delta}+ V_n \nabla u_n^{\delta} + u_n^{\delta} \nabla V_n+\lambda \nabla u_n^{\delta} \log\big(\delta+|u_n^{\delta}|^2\big) +  2 \lambda  \Re \left(\overline{u_n^{\delta}} \nabla u_n^{\delta} \right)\frac{u_n^{\delta}}{\delta+|u_n^{\delta}|^2}
\end{equation*}
which after taking the imaginary part of the scalar product with $\nabla u_n^{\delta}$ gives
\begin{align*}
\frac12 \frac{\dd}{\dd t} \| \nabla u_n^{\delta}(t)\|_{L^2}^2 & \leq |\lambda| \int_{\R^d}\frac{|u_n^{\delta}(x)|^2}{\delta+|u_n^{\delta}(x)|^2} |\nabla u_n^{\delta}(x)|^2\drm x+\Big|\int_{\R^d}\Im\big(\nabla V_n(x) u_n^{\delta}(x) \overline{\nabla u_n^{\delta}}(x)\big)\drm x\Big|\\
&\le\Big( |\lambda| + \|\nabla V_n\|_{L^{\infty}} \| u_n^{\delta}(t) \|^2_{L^2}\Big) \| \nabla u_n^{\delta}(t)\|_{L^2}^2\\
&\le\Big( |\lambda| + \|\nabla V_n\|_{L^{\infty}} \| u_0 \|^2_{L^2}\Big) \| \nabla u_n^{\delta}(t)\|_{L^2}^2
\end{align*}
using the conservation of mass. Gronwall Lemma then yields the bound
\begin{equation*}
\sup_{t\in I}\| \nabla u_n^{\delta}(t)\|_{L^2}^2 \le e^{2|I|\big(|\lambda|+\|\nabla V_n\|_{L^{\infty}}\|u_0\|^2_{L^2}\big)}\|\nabla u_0\|_{L^2}^2
\end{equation*}
for any finite interval $I\subset\R$, which is uniform in $\delta>0$ but may diverge in $n\ge1$. We now obtain a bound in $L_{\mu_0}^2$ for $u$. We have
\begin{align*}
\frac12\frac{\dd }{\dd t}\|u_n^{\delta}(t)\|_{L^2_{\mu_0} }^2&=\mu_0\int_{\R^d}\Im\Big(\frac{x \cdot \nabla u_n^{\delta}(x)}{\langle x \rangle^{2-2\mu_0}}  \overline{u_n^{\delta}(x)}\Big) \dd x \\
&\le\mu_0 \| u_n^{\delta}(t) \|_{L^2} \| \nabla u_n^{\delta}(t) \|_{L_{2\mu_0-1}^2}\\
&\le\mu_0\| u_0 \|_{L^2} \|u_n^{\delta}(t)\|_{\CH^1}
\end{align*}
since $2\mu_0\le1$ and using the mass conservation. This gives
\begin{equation*}
\sup_{t\in I}\|u_n^\delta(t)\|_{L_{\mu_0}^2}^2\le\|u_0\|_{L_{\mu_0}^2}^2+2|I|\mu_0\|u_0\|_{L^2}\sup_{t\in I}\|u_n^\delta(t)\|_{\CH^1}
\end{equation*}
for any finite interval $I\subset\R$ hence a bound of $u_n^\delta(t)$ in $L_{\mu_0}^2$ uniform in $\delta>0$ with the previous bound in $\CH^1$. We now need to take the limit $\delta$ to $0$ for the approximating sequence $(u_n^{\delta})_{\delta}$ in order to recover a weak solution $u_n$. This follows closely the compactness arguments from the seminal paper of Ginibre and Velo \cite{ginibre1985}, we give its main ingredients as this strategy will be used throughout this paper. From the uniform bounds in $\CH^1$ and $L^2_{\mu_0}$, we infer that up to extraction, not relabeled for reader's convenience, 
\begin{equation*}
u_n^{\delta} \rightharpoonup u_n \quad \text{in} \quad \CC(I,\CH^1 \cap L^2_{\mu_0})
\end{equation*}
and $\| u_n \|_{\CC(I,\CH^1 \cap L^2_{\mu_0})} \leq C \|u_0 \|_{\CH^1 \cap L^2_{\mu_0}}$ where $C=C(n,|I|,\|u_0\|_{L^2})>0$. From Lemma \ref{logarithmic_lemma}, we also get that $u^{\delta}_n \log(|u^{\delta}_n|^2+\delta)$ is uniformly bounded in $\CC(I,L^2)$ so that
\begin{equation*}
u_n^{\delta}\log(|u^{\delta}_n|^2+\delta )\rightharpoonup F \quad \text{in} \quad \CC(I,L^2)
\end{equation*}
and $\| F \|_{\CC(I,L^2)} \leq C \|u_0\|_{\CH^1 \cap L^2_{\mu_0}}$ for some $F\in\CC(I,L^2)$, again up to extraction. From the compact embedding $\CH^1 \cap L^2_{\mu_0} \hookrightarrow L^1$ and by passing to the limit $\delta$ to $0$ into the weak formulation of the regularized equation, we get
\begin{equation*} 
i \partial_t u_n = \Delta u_n +V_n u_n + \lambda F
\end{equation*}
in the sense of distributions since $V_n\in\CC^\infty$. It only remains to show that $F=u_n \log (|u_n|^2)$ which is the technical part of this compactness method.  We have
\begin{align*}
\|u(t_2)-u(t_1)\|_{L^2}^2&=\int_{t_1}^{t_2}2\Re\big\langle u(t)-u(t_1),\partial_tu(t)\big\rangle\drm t\\
&\le4|t_2-t_1|\|u\|_{\CC(I,\CH^1)}|\|\partial_tu\|_{\CC(I,\CH^{-1})}
\end{align*}
which gives a uniform bound in $\CC^{\frac12}(I,L^2)$ hence uniform equicontinuity. Restricting our attention to a compact $K\subset\IR^d$, we also naturally get a uniform bound in $\CC\big(I,\CH^1(K)\big)$. The compactness of the embedding $\CH^1(K) \hookrightarrow L^2(K)$ ensures that the set $\enstq{ u^{\delta}_n(t)}{\delta \in (0,1)}$ is relatively compact in $L^2(K)$. Up to extraction, one gets convergence in $C\big(I,\CH^1(K)\big)$ from Arzelà-Ascoli's Theorem, see for example Ginibre's lecture notes \cite[Lemma 7.7]{ginibre1995}. Since we already know that $(u^{\delta}_n)_{\delta}$ converges weakly to $u_n$ in $\CC\big(I,L^2(K)\big)$, we have
\begin{equation*}
u_n^{\delta} \rightarrow  u_n \quad \text{in} \quad \CC\big(I,L^2(K)\big)
\end{equation*}
for all finite $I \subset \R$, hence $u_n \in \CC\big(\R,L^2(K)\big)$ and $u_n(0)=u_0$. Finally, up to extraction, $u_n^{\delta}(t,x)\rightarrow u_n(t,x)$ for almost every $(t,x) \in \R \times K$, therefore
\begin{equation*}
u_n^{\delta}(t,x) \log (|u_n^{\delta}(t,x)|^2 + \delta) \rightarrow u_n(t,x) \log |u_n(t,x)|^2.
\end{equation*}
As we also have weak convergence towards $F$ and the compact $K$ is arbitrary, we infer that
\begin{equation*}
F(t,x)=u_n(t,x) \log |u_n(t,x)|^2
\end{equation*}
for almost all $(t,x) \in \R \times \R^d$, which ensures that $u_n$ is indeed a solution to
\begin{equation*}
i \partial_t u_n=\Delta u_n+V_n u_n+\lambda u_n\log |u_n|^2
\end{equation*}
with conserved mass and conserved energy
\begin{equation*}
E_n(u_n)=\int_{\R^d} |\nabla u_n(x)|^2\drm x- \int_{\R^d} V_n(x)|u_n(x)|^2\drm x  - \lambda \int_{\R^d}|u_n(x)|^2\log|u_n(x)|^2\drm x.
\end{equation*}
It only remains to pass to the limit $n$ to $+\infty$. For the potential, we have
\begin{equation*}
\Big|\int_{\IR^d}V_n(x)|u_n(t,x)|^2\drm x\Big|\le\|V\|_{L_{-\mu_0}^\infty}\|u_n(t)\|_{L_{\mu_0}^2}^2
\end{equation*}
which gives with the previous bound
\begin{equation*}
\sup_{t\in I}\Big|\int_{\IR^d}V_n(x)|u_n(t,x)|^2\drm x\Big|\le\|V\|_{L_{-\mu_0}^\infty}\|u_0\|_{L_{\mu_0}^2}^2+2|I|\mu_0\|V\|_{L_{-\mu_0}^\infty}\|u_0\|_{L^2}\sup_{t\in I}\|u_n^\delta(t)\|_{\CH^1}
\end{equation*}
since $2\mu_0\le 1$. For the logarithmic term, we have
\begin{equation*}
\Big|\int_{\R^d}|u_n(t,x)|^2\log|u_n(t,x)|^2\drm x\Big|\le C\|u_n(t)\|_{L_{\mu_0}^2}^{\frac{d\eta}{2\mu_0}}\|u_0\|_{L^2}^{2-\eta-\frac{d\eta}{2\mu_0}}+C\|u_n(t)\|_{\CH^1}^{\frac{d\eta}{2}}\|u_0\|_{L^2}^{2-\eta-\frac{d\eta}{2}}
\end{equation*}
for any $\eta>0$ small enough using Lemma \ref{logarithmic_lemma} with $\mu=0$ and the conservation of mass. With the previous bound on $\|u_n(t)\|_{L_{\mu_0}^2}$, we get
\begin{equation*}
\sup_{t\in I}\Big|\int_{\R^d}|u_n(t,x)|^2\log|u_n(t,x)|^2\drm x\Big|\le A_1+A_2|I|^{\frac{d\eta}{2\mu_0}}\sup_{t\in I}\|u_n(t)\|_{\CH^1}^{\frac{d\eta}{4\mu_0}}+A_3\sup_{t\in I}\|u_n(t)\|_{\CH^1}^{\frac{d\eta}{2}}
\end{equation*}
with $A_1,A_2,A_3>0$ constant depending on $d,\eta$ and $\|u_0\|_{L_{\mu_0}^2}$. Using the conservation of energy, we get
\begin{align*}
\sup_{t\in I}\Big|\int_{\R^d}|\nabla u_n(t,x)|^2\drm x\Big|&\le\big|E_n(u_0)\big|+\|V\|_{L_{-\mu_0}^\infty}\|u_0\|_{L_{\mu_0}^2}^2+4|I|\mu_0\|V\|_{L_{-\mu_0}^\infty}\|u_0\|_{L^2}\sup_{t\in I}\|u_n^\delta(t)\|_{\CH^1}\\
&\quad+A_1+A_2|I|^{\frac{d\eta}{2\mu_0}}\sup_{t\in I}\|u_n(t)\|_{\CH^1}^{\frac{d\eta}{4\mu_0}}+A_3\sup_{t\in I}\|u_n(t)\|_{\CH^1}^{\frac{d\eta}{2}}
\end{align*}
which, for $\eta\le\frac{4\mu_0}{d}$, gives
\begin{equation*}
\sup_{t\in I}\|\nabla u_n(t)\|_{L^2}^2\le A+B\sup_{t\in I}\|u_n(t)\|_{\CH^1}
\end{equation*}
with constant $A,B>0$ depending on $d,\mu_0,I,\|u_0\|_{L_{\mu_0}^2}$ and $E(u_0)$. This yields
\begin{equation*}
\sup_{n\ge1}\ \sup_{t\in I}\|\nabla u_n(t)\|_{L^2}^2<\infty
\end{equation*}
and allows to pass to the limit in $n$ to $+\infty$ in $\CH^1$ using Arzelà-Ascoli's Theorem as before and complete the proof of the existence. The uniqueness follows from the inequality
\begin{equation} \label{log_equation_cazenave} \tag{$\star$}
\Big|\Im\big(z\log|z|^2-z'\log|z'|^2\big)(\overline z-\overline{z'})\Big|\le4|z-z'|^2
\end{equation}
for all $z,z'\in\C$, see for example \cite[Lemma 9.3.5]{cazenave}. Indeed, for $u_1,u_2$ two solutions, $v=u_1-u_2$ satisfies
\begin{equation*}
i\partial_tv=\Delta v+Vv+\lambda\big(u_1\log|u_1|^2-u_2\log|u_2|^2\big)
\end{equation*}
hence
\begin{align*}
\frac{1}{2}\frac{\drm}{\drm t}\|v(t)\|_{L^2}^2&\le|\lambda|\int_{\IR^d}\Big|\Im\big(u_1(x)\log|u_1(x)|^2-u_2(x)\log|u_2(x)|^2\big)\big(\overline{u_1(x)}-\overline{u_2(x)}\big)\Big|\drm x\\
&\le4|\lambda|\|v(t)\|_{L^2}^2
\end{align*}
which gives $v=0$ using Gronwall Lemma.

\medskip

For $u_0\in\CH^2$, again the usual theory gives a solution $u_n^\delta\in\CC(\IR,\CH^2)$ to the truncated and regularized equation. The time derivative $v_n^\delta:=\partial_tu_n^\delta$ satisfies
\begin{equation*}
i\partial_t v_n^{\delta} = \Delta v_n^{\delta} +V_n v_n^{\delta}+ \lambda v_n^{\delta} \log\big(\delta+|u_n^{\delta}|^2\big)+2\lambda \Re \left( \overline{u_n^{\delta}} v_n^{\delta}  \right)  \frac{u_n^{\delta}}{\delta+|u_n^{\delta}|^2}
\end{equation*}
with initial data $v_n^\delta(0)=\partial_tu_n^\delta(0)=\Delta u_0+V_nu_0+\lambda u_0\log\big(\delta+|u_0|^2\big)$. Following the same path as before gives
\begin{align*}
\frac12 \frac{\dd}{\dd t} \| v_n^{\delta}(t) \|_{L^2}^2 &\le |\lambda|\Big|\int_{\R^d} \Re \big( \overline{u_n^{\delta}(t,x)}v_n^{\delta}(t,x)  \big) \frac{\Im \big(u_n^{\delta}(t,x) \overline{v_n^{\delta}(t,x)}  \big)}{\delta+|u_n^{\delta}(t,x)|^2}\drm x \Big|\\
&\le |\lambda| \|v_n^{\delta}(t) \|_{L^2}^2
\end{align*}
hence Gronwall lemma gives
\begin{equation*}
\|v_n^\delta(t)\|_{L^2}^2\le e^{2|\lambda t|}\|\Delta u_0+V_nu_0+\lambda u_0\log\big(\delta+|u_0|^2)\|_{L^2}^2
\end{equation*}
using the expression for $v_n^\delta(0)$. Since $u_0\in L_{\mu_0}^2\cap\CH^2$, we have
\begin{equation*}
\|\Delta u_0+V_nu_0\|_{L^2}<\infty
\end{equation*}
hence we only have to deal with the logarithmic term. It follows from the same proof as for Proposition \ref{logarithmic_lemma} with
\begin{equation*}
|u_0|^2\log(\delta+|u_0|^2)\le C(|u_0|^{2+\eta}+|u_0|^{2-\eta})
\end{equation*}
with $\eta>0$ arbitrary small and $C>0$ a constant independent of $\delta\in(0,1)$. In the end, we get
\begin{equation*}
\sup_{n\ge1}\ \sup_{\delta>0}\ \sup_{t\in I}\|v_n^\delta(t)\|_{L^2}^2<\infty
\end{equation*}
and, with a similar bound on the potential and logarithmic term, the equation gives
\begin{equation*}
\sup_{n\ge1}\ \sup_{\delta>0}\ \sup_{t\in I}\|u_n^\delta(t)\|_{\CH^2}<\infty
\end{equation*}
which completes the proof.
\end{proof}

\begin{remark}
Note that $\CH_\mu^1\subsetneq L_\mu^2\cap\CH^1$ for $\mu>0$ since $u\in\CH_\mu^1$ is equivalent to $x\mapsto\langle x\rangle^\mu u(x)\in\CH^1$ thus
\begin{equation*}
\|\langle x\rangle^{\mu-1}u(x)+\langle x\rangle^\mu\nabla u(x)\|_{L^2}<\infty
\end{equation*}
which is a priori not true for $u\in L_\mu^2\cap\CH^1$. However for any $\gamma,\mu\in\R$, Lemma \ref{interpolation_estimate_Besov} gives
\begin{equation*}
\|u\|_{\CH_{(1-\theta)\mu}^{\gamma\theta}}\le C\|u\|_{L_\mu^2}^{1-\theta}\|u\|_{\CH^\gamma}^\theta
\end{equation*}
with $C>0$ a constant and $\theta\in(0,1)$. In particular, the previous result implies
\begin{equation*}
u\in \CC(\R,\CH_{(1-\theta)\mu_0}^\theta)\quad\text{and}\quad u\in \CC(\R,\CH_{(1-\theta)\mu_0}^{2\theta})
\end{equation*}
for any $\theta\in(0,1)$ respectively for $u_0\in L_{\mu_0}\cap\CH^1$ and $u_0\in L_{\mu_0}\cap\CH^2$. This will be used often in the following.
\end{remark}

\section{Stochastic equation on the line} \label{Section1D}

The stochastic logarithmic Schrödinger equation 
\begin{equation} \label{SlogNLS_R} \tag{SlogNLS}
i\partial_tu =\Delta u+\xi u+\lambda u\log |u|^2 
\end{equation} 
on $\R$ has two conservation laws, namely the mass conservation
\begin{equation*}
M(u)=\int_\R|u(x)|^2\dd x
\end{equation*}
and the energy conservation
\begin{equation*}
E(u)= \int_\R|\nabla u(x)|^2\dd x - \lambda \int_\R|u(x)|^2\log|u(x)|^2\dd x - \int_\R|u(x)|^2\xi(x)\dd x.
\end{equation*}
Since the noise is a distribution, the last term in the energy has to be interpreted as a distribution bracket which is well-defined for smooth enough $u$. In one dimension, $\xi\in\CC_{-\mu}^{-\frac{1}{2}-\kappa}$ for any $\kappa,\mu>0$ hence the bracket is well-defined for $u\in\CH_{\mu'}^{\frac{1}{2}+\kappa'}$ with $\mu<\mu'$ and $\kappa<\kappa'$. In particular, the energy is well-defined for any $u\in\CH_{\mu_0}^1$ with $\mu_0>0$. Because of the irregularity of the noise, one can not hope to propagate arbitrary smoothness for the initial data, we prove the following theorem. Recall that the random field $X$ introduced before satisfies the equation
\begin{equation*}
\Delta X=\xi+\varphi*\xi
\end{equation*}
with $\varphi\in\CC_c^\infty$ a smooth compactly supported function and that $X,e^X,e^{-X}\in\CC_{-\mu}^{-\frac32-\kappa}$ for any $\mu,\kappa>0$.

\medskip

\begin{theorem}
Let $0<\mu_0\le\frac12$ and $u_0\in L_{\mu_0}^2(\R)\cap\CH^1(\R)$. There exists a unique solution $u \in \mathcal{C}\big(\R,L_{\mu_0}^2(\R)\cap\CH^1(\R)\big)$ to \eqref{SlogNLS_R} with initial data $u_0$. If moreover $u_0\in e^{-X}\CH_{\mu_0}^2$, then $u\in\CC\big(\R,e^{-X}\CH_{(1-\theta)\mu}^{2\theta}\big)$ for any $\mu<\mu_0$ and $\theta\in(0,1)$. In particular, the solution belongs to $\CC\big(\R,\CH_{(1-\frac{\gamma}{2})\mu}^\gamma\big)$ for any $\mu<\mu_0$ and $\gamma<\frac32$.
\end{theorem}

\medskip

\begin{proof}
From Proposition \ref{existence_logNLS_R_prop}, there exists a global solution $u_\eps\in\CC(\R,L_{\mu_0}^2\cap\CH^1)$ to
\begin{equation*}
i\partial_tu_\eps=\Delta u_\eps+\xi_\eps u_\eps+\lambda u_\eps\log|u_\eps|^2
\end{equation*}
with initial data $u_0\in L_{\mu_0}^2\cap\CH^1$ with conserved mass and regularized energy
\begin{equation*}
E_\eps(u_\eps)= \int_\R|\nabla u_\eps(x)|^2\dd x- \int_\R|u_\eps(x)|^2\xi_\eps(x)\dd x - \lambda \int_\R|u_\eps(x)|^2\log|u_\eps(x)|^2\dd x.
\end{equation*}
For $\eps$ small enough, we have
\begin{equation*}
\int_\R|\nabla u_\eps(t,x)|^2\dd x\le2|E(u_0)|+\Big|\lambda\int_\R|u_\eps(t,x)|^2\log|u_\eps(t,x)|^2\dd x\Big|+\Big|\int_\R|u_\eps(t,x)|^2\xi_\eps(x)\dd x\Big|
\end{equation*}
with $|E(u_0)|<\infty$ since $u_0\in L_{\mu_0}^2\cap\CH^1\subset\CH_{(1-\theta)\mu_0}^\theta$ for any $\theta\in(0,1)$. For the logarithmic term, the proof of Proposition \ref{existence_logNLS_R_prop} gives
\begin{equation*}
\sup_{t\in I}\Big|\int_{\R^d}|u_\eps(t,x)|^2\log\big(|u_\eps(t,x)|^2\big)\drm x\Big|\le A_1+A_2|I|^{\frac{\eta}{2\mu_0}}\sup_{t\in I}\|u_\eps(t)\|_{\CH^1}^{\frac{\eta}{4\mu_0}}+A_3\sup_{t\in I}\|u_\eps(t)\|_{\CH^1}^{\frac{\eta}{2}}
\end{equation*}
with $A_1,A_2,A_3>0$ constant depending on $\eta,\mu_0,\|u_0\|_{L_{\mu_0}^2}$ and $\|u_0\|_{\CH^1}$ for any finite interval $I\subset\R$. For the noise term, we have
\begin{align*}
\left|\int_\R\xi_\eps(x)|u_\eps(t,x)|^2\dd x\right|&\le C\|\xi_\eps\|_{\CC_{-2\mu}^{-\frac{1}{2}-\kappa}}\|u_\eps(t)\overline{u_\eps(t)}\|_{\CB_{1,1,2\mu}^{\frac{1}{2}+\kappa}}\\
&\le C\|\xi\|_{\CC_{-2\mu}^{-\frac{1}{2}-\kappa}}\|u_\eps(t)\|_{\CH_\mu^{\frac{1}{2}+2\kappa}}^2\\
&\lesssim C'\|\xi\|_{\CC_{-2\mu}^{-\frac{1}{2}-\kappa}}\|u_\eps(t)\|_{L_{\frac{\mu}{\frac{1}{2}-2\kappa}}^2}^{1-4\kappa}\|u_\eps(t)\|_{\CH^1}^{1+4\kappa}
\end{align*}
for $\kappa,\mu>0$ using Lemma \ref{product_estimate_Besov} and Lemma \ref{interpolation_estimate_Besov} with $C'>0$ a constant depending on $\mu$ and $\kappa$. Again following the proof of Proposition \ref{existence_logNLS_R_prop}, we have
\begin{equation*}
\sup_{t\in I}\|u_\eps(t)\|_{L_{\mu_0}^2}^2\le\|u_0\|_{L_{\mu_0}^2}^2+2|I|\mu_0\|u_0\|_{L^2}\sup_{t\in I}\|u_\eps(t)\|_{\CH^1}
\end{equation*}
hence for $\mu=\mu_0(\frac{1}{2}-2\kappa)$, we get
\begin{equation*}
\sup_{t\in I}\left|\int_\R\xi_\eps(x)|u_\eps(t,x)|^2\dd x\right|\le C\sup_{t\in I}\big(\|u_\eps(t)\|_{H^1}^{1+4\kappa}+|I|^{\frac{1}{2}-2\kappa}\|u_\eps(t)\|_{H^1}^{\frac{3}{2}+2\kappa}\big)
\end{equation*}
with $C>0$ a constant depending on $\mu_0,\kappa$ and $\|u_0\|_{L_{\mu_0}^2}$. Thus the conservation of energy gives
\begin{equation*}
\sup_{t\in I}\int_\R|\nabla u_\eps(t,x)|^2\dd x\le 2|E(u_0)|+C\sup_{t\in I}\big(\|u_\eps(t)\|_{H^1}^{1+4\kappa}+|I|^{\frac{1}{2}-2\kappa}\|u_\eps(t)\|_{H^1}^{\frac{3}{2}+2\kappa}\big)
\end{equation*}
which for $\frac{3}{2}+2\kappa<2$ gives
\begin{equation*}
\sup_{\eps>0}\ \sup_{t\in I}\|u_\eps(t)\|_{\CH^1}<\infty.
\end{equation*}
This bound gives the existence of a solution using Arzelà-Ascoli’s Theorem as before, the noise part being well controlled since $\xi\in\CC_{-\mu}^{-\frac12-\kappa}$ and that the solution is controlled in $L_{\mu_0}^2\cap\CH^1$. The proof of uniqueness is the same as in the deterministic case of Proposition \ref{existence_logNLS_R_prop}.

\medskip

For the propagation of regularity, the usual path is to suppose higher Sobolev regularity for the initial data and prove that it is conserved for positive time. Because of the roughness of the noise $\xi$, this does not hold anymore. However, one can consider initial data in the Sobolev space associated to the Anderson Hamiltonian. Consider the new variable
\begin{equation*}
v:=e^Xu
\end{equation*} 
where $X$ is a solution to
\begin{equation*}
\Delta X=\xi+\varphi*\xi
\end{equation*}
with $\varphi\in\CC_c^\infty(\R)$ is a smooth compactly supported function introduced before. It satisfies the equation
\begin{equation*}
i\partial_tv=\Delta v-2\nabla X\cdot \nabla v-2\lambda Xv+v(|\nabla X|^2-\varphi*\xi)+\lambda v\log|v|^2
\end{equation*}
with initial data $v_0=e^X u_0$. While this equation seems more complicated, it is better behaved since the roughest term is canceled. With our previous result, one gets global existence and uniqueness for initial data $v_0\in e^X(L_{\mu_0}^2\cap\CH^1)$ with $v\in\CC\big(\R,e^X(L_{\mu_0}^2\cap\CH^1)\big)$. Since $e^X,e^{-X}\in\CC_{-\mu}^{\frac{3}{2}-\kappa}$ for any $\mu,\kappa>0$, we have
\begin{equation*}
e^X(L_{\mu_0}^2\cap\CH^1)\subset \CH_{(1-\theta)\mu}^\theta
\end{equation*}
for any $\mu<\mu_0$ and $\theta\in(0,1)$ as well as
\begin{equation*}
\CH_{\mu_0}^1\subset e^X(L_\mu^2\cap\CH^1)
\end{equation*}
for any $\mu<\mu_0$. For $v_0\in\CH_{\mu_0}^1$, we then have $e^{-X}v_0\in L_\mu^2\cap\CH^1$ for $\mu<\mu_0$ thus the previous result implies the existence of a global solution $v\in\CC(\R,\CH_{(1-\theta)\mu}^\theta)$ for any $\mu<\mu_0$ and $\theta\in(0,1)$. Using Kato's trick, we now prove the propagation of regularity for $v_0\in\CH_{\mu_0}^2$. Let $w:=\partial_tv$ which formally satisfies the equation
\begin{equation*}
i\partial_tw=\Delta w-2\nabla X \cdot \nabla w-2\lambda Xw+w(|\nabla X|^2-\varphi*\xi)+\lambda w\log|v|^2+2\lambda\Re(\overline vw)\frac{v}{|v|^2}
\end{equation*}
with initial data 
\begin{equation*}
w(0)=(\partial_tv)(0)=\Delta v_0-2\nabla X \cdot \nabla v_0-2\lambda Xv_0+v_0(|\nabla X|^2-\varphi*\xi)+\lambda v_0\log|v_0|^2.
\end{equation*}
Of course, we point out that the term $2\lambda\Re(\overline vw)v/|v|^2$ might be singular, and one has to perform the same kind of regularization procedure with saturating constant $\delta>0$ as in the proof of Proposition~\ref{existence_logNLS_R_prop} and then to pass to the limit $\delta \rightarrow 0$ thanks to uniform estimates. As it follows the exact same path as we did before, we omit the details here. Since $X\in\CC_{-\mu}^{\frac{3}{2}-\kappa}$ for any $\mu,\kappa>0$ and $v_0\in\CH_{\mu_0}^2$, we get that $e^{-X}w(0)\in L^2$, which is exactly where it is important that we have canceled the roughest term $\xi$, and we can then follow the usual path. Using the equation at the level of $u$, Kato's trick gives
\begin{equation*}
\|e^{-X}w(t)\|_{L^2}^2\le e^{4|\lambda t|}\|e^{-X}w(0)\|_{L^2}^2
\end{equation*}
thus
\begin{equation*}
\|w(t)\|_{L_{-\mu}^2}^2\le e^{4|\lambda t|}\|e^X\|_{L_{-\mu}^\infty}\|e^{-X}w(0)\|_{L^2}^2
\end{equation*}
for any $\mu>0$. Using the equation on $v$, we get
\begin{align*}
\|\Delta v(t)\|_{L_{-\mu}^2}&\le\|i\partial_tv+2\nabla X \cdot \nabla v+2\lambda Xv-v(|\nabla X|^2-\varphi*\xi)+\lambda v\log|v|^2\|_{L_{-\mu}^2}\\
&\le e^{2|\lambda t|}\|e^X\|_{L_{-\mu}^\infty}^{\frac12}\|e^{-X}w(0)\|_{L^2}^2+2C\|\nabla X\|_{L_{-\mu}^\infty}\|\nabla v\|_{L_{2\mu}^2}+2C|\lambda|\|X\|_{L_{-\mu}^\infty}\|v\|_{L_{2\mu}^2}\\
&\quad+C\||\nabla X|^2-\varphi*\xi\|_{L_{-\mu}^\infty}\|v\|_{L_{2\mu}^2}+C\|v\|_{L^2_{\mu_0}}^{\frac{\eta}{2\mu_0}}\|v\|_{L^2}^{2-\eta-\frac{\eta}{2\mu_0}}+\|v\|_{H^1}^{\frac{\eta}{2}}\|v\|_{L^2}^{2+\eta-\frac{\eta}{2}}
\end{align*}
with a parameter $\eta>0$ hence $v\in \CC \big(\R,\CH_{-\mu}^2\big)$ for any $\mu>0$ small enough. Using that
\begin{equation*}
L_\mu^2\cap\CH_{-\mu'}^2\hookrightarrow\CH_{(1-\theta)\mu-\theta\mu'}^{2\theta}
\end{equation*}
for any $\theta\in(0,1)$ and $\mu'>0$, this proves that $v\in \CC \big(\R,\CH_{(1-\theta)\mu}^{2\theta}\big)$ for any $\theta\in(0,1)$ and $\mu<\mu_0$. In particular, this proves that the condition $u_0\in e^{-X}\CH_{\mu_0}^2$ implies $u\in \CC \big(\R,e^{-X}\CH_{(1-\theta)\mu}^{2\theta}\big)$ for any $\theta\in(0,1)$ and $\mu<\mu_0$ which completes the proof.
\end{proof}

\section{Stochastic equation on the plane} \label{Section2D}

In two dimensions, the noise is even more irregular and the equation becomes singular in the sense of ill-defined product. Indeed, assuming that one can construct a solution in $\CH_{\mu_0}^1$, the equation
\begin{equation*}
i\partial_tu=\Delta u+u\xi+\lambda u\log|u|^2
\end{equation*}
contains an ill-defined product since $\xi\in\CC_{-\mu}^{-1-\kappa}(\R^2)$ for any $\mu,\kappa>0$. As illustrated with the equation in one dimension, it is not expected that regularity higher than $\beta+2$ with $\beta$ being the regularity on the noise can be propagated for the solution. As done in the context of dispersive singular SPDE by Debussche and Weber \cite{debussche2018}, a way out is to consider the exponential transformation, as done in the previous section, that is $v=e^Xu$ with
\begin{equation*}
\Delta X=\xi+\varphi*\xi
\end{equation*}
for a function $\varphi\in\CC_c^\infty(\R^2)$. Then $v$ formally satisfies the equation
\begin{equation*} \label{shifted_logarithmic_eq_R2}
i \partial_t v =\Delta v - 2\nabla X\cdot\nabla v+v ( |\nabla X|^2-\varphi \ast \xi) -2\lambda Xv+ \lambda v \log|v|^2
\end{equation*}
with initial data $v_0=e^Xu_0$. This equation is still singular since in two dimensions, we have $X\in\CC_{-\mu}^{1-\kappa}$ hence $\nabla X$ is a distribution and the square $|\nabla X|^2$ is ill-defined. This is however a better behaved equation since the roughest term is canceled and the singular product is independant of the unknown $v$ as it concerns only $X$. For a regularization of the noise $\xi_\eps$, the divergence of the singular term $|\nabla X_\eps|^2$ is well known with the Wick product
\begin{equation*}
\Wick{|\nabla X|^2}\ =\lim_{\eps\to0}\Big(|\nabla X_\eps|^2-\IE\big[|\nabla X_\eps|^2\big]\Big)
\end{equation*}
in the space $\CC_{-\mu}^{-\kappa}(\R^2)$ for any $\mu,\kappa>0$. In particular, the mean $c_\eps:=\IE\big[|\nabla X_\eps|^2\big]$ diverges as a logarithm of $\eps$. Hence one can hope to prove that the solution $v_\eps$ of the modified equation
\begin{equation*}
i \partial_t v_\eps =\Delta v_\eps - 2\nabla X_\eps\cdot\nabla v_\eps+v_\eps( |\nabla X_\eps|^2-c_\eps-\varphi \ast \xi_\eps)-2\lambda X_\eps v_\eps + \lambda v_\eps \log|v_\eps|^2
\end{equation*}
converges as $\eps$ goes to $0$ to the solution of
\begin{equation*}\label{SLNLSmodif}\tag{mSlogNLS}
i \partial_t v=\Delta v-2\nabla X\cdot\nabla v+v(\Wick{|\nabla X|^2}-\varphi*\xi)-2\lambda Xv+\lambda v\log|v|^2
\end{equation*}
with initial data $v_0\in\CH_{\mu_0}^2$. The initial unknown $u_\eps$ satisfies the equation
\begin{equation*}
i\partial_tu_\eps=\Delta u_\eps+u_\eps(\xi_\eps-c_\eps)+\lambda u_\eps\log|u_\eps|^2
\end{equation*}
with initial data $u_\eps(0)=e^{-X_\eps}v_0$ and has conserved mass and conserved energy
\begin{equation*}
\int_{\R^2}|\nabla u_\eps(x)|^2\dd x - \int_{\R^2}|u_\eps(x)|^2(\xi_\eps(x)-c_\eps)\dd x- \lambda \int_{\R^2}|u_\eps(x)|^2\log|u_\eps(x)|^2\dd x.
\end{equation*}
In the new variable $v$, this gives the modified energy
\begin{align*}
E_\eps(v_\eps)&=\int_{\R^2}|\nabla v_\eps(x)|^2e^{-2X_\eps(x)}\dd x-\int_{\IR^2}|v_\eps(x)|^2(|\nabla X_\eps|^2-c_\eps-\varphi*\xi_\eps)e^{-2X_\eps(x)}\dd x \\
&\quad- \lambda \int_{\R^2}|v_\eps(x)|^2\log|v_\eps(x)|^2e^{-2X_\eps(x)}\dd x+2\lambda\int_{\R^2}|v_\eps(x)|^2X_\eps(x)e^{-2X_\eps(x)}\dd x
\end{align*}
and modified mass
\begin{equation*}
M_\eps(v_\eps)=\int_{\R^2}|v_\eps(x)|^2e^{-2X_\eps(x)}\dd x.
\end{equation*}
While this seems more complicated, this converges to the modified energy $E(v)$ and mass $M(v)$, well-defined for $v\in\CH_{\mu_0}^2$. Given a solution $v_\eps$ which converges to a solution $v$ of \eqref{SLNLSmodif}, we can interprete $u=e^{-X}v$ as a solution of the renormalized equation of \eqref{SlogNLS} and limit of $u_\eps=e^{-X_\eps}v_\eps$ in a suitable weighted space.

\medskip

\begin{theorem}
Let $v_0\in\CH_{\mu_0}^2(\R^2)$ with $0<\mu_0\le\frac13$. There exists a unique solution $v\in\CC\big(\R,\CH_{(1-\theta)\mu}^{2\theta}(\R^2)\big)$ of equation \eqref{SLNLSmodif} for any $\mu<\mu_0$ and $\theta\in(0,1)$.
\end{theorem}

\medskip

\begin{proof}
Since $v_0\in\CH_{\mu_0}^2$, we have $e^{-X_\eps}v_0\in L_\mu^2\cap\CH^2$ for any $\mu<\mu_0$. From Proposition \ref{existence_logNLS_R_prop}, there exists a global solution $u_\eps\in\CC(\R,L_\mu^2\cap\CH^2)$ with $\mu<\mu_0$ to
\begin{equation*}
i\partial_tu_\eps=\Delta u_\eps+(\xi_\eps-c_\eps)u_\eps+\lambda u_\eps\log|u_\eps|^2
\end{equation*}
with initial data $e^{-X_\eps}v_0$ with conserved mass $M$ and regularized energy
\begin{equation*}
\int_\R|\nabla u_\eps(x)|^2\dd x- \int_\R|u_\eps(x)|^2(\xi_\eps(x)-c_\eps)\dd x - \lambda \int_\R|u_\eps(x)|^2\log|u_\eps(x)|^2\dd x.
\end{equation*}
This implies that $v_\eps\in\CC\big(\R,\CH_{(1-\theta)\mu}^{2\theta}\big)$ for any $\theta\in(0,1)$ with conserved modified mass
\begin{equation*}
M_\eps(v_\eps)=\int_{\R^2} |v_{\eps}(x)|^2 e^{-2X_{\eps}(x)} \dd x
\end{equation*}
and modified energy
\begin{align*}
E_\eps(v_\eps)&=\int_{\R^2}|\nabla v_{\eps}(x)|^2e^{-2X_\eps(x)}\dd x-\int_{\R^2}|v_{\eps}(x)|^2\big(|\nabla X_\eps(x)|^2+(\varphi*\xi_\eps)(x)-c_\eps\big)e^{-2X_\eps(x)}\dd x\\
&\quad+2\lambda\int_{\IR^2}|v_{\eps}(x)|^2 X_\eps(x) e^{-2X_{\eps}(x)} \dd x-\lambda\int_{\IR^2}|v_{\eps}(x)|^2 \log |v_{\eps}(x)|^2e^{-2X_\eps(x)}\dd x.
\end{align*}
For $\mu\in\R$, we have
\begin{align*}
\frac{1}{2}\frac{\drm}{\drm t}\|v_\eps e^{-X_\eps}\|_{L_\mu^2}^2&=\int_{\IR^2}\langle x\rangle^{2\mu}\Re\big(\partial_tv_\eps(t,x)\overline{v_\eps}(t,x)\big)e^{-2X_\eps(x)}\drm x\\
&=\int_{\IR^2}\langle x\rangle^{2\mu}\Im\Big(\big(\Delta v_\eps(t,x)-2\nabla X_\eps(x)\cdot\nabla v_\eps(t,x)\big)\overline{v_\eps}(t,x)\Big)e^{-2X_\eps(x)}\drm x\\
&=-\int_{\IR^2}\nabla(\langle x\rangle^{2\mu})\cdot\nabla v_\eps(t,x)\overline{v_\eps}(t,x)e^{-2X_\eps(x)}\drm x\\
&\le 2|\mu|\int_{\IR^2}\langle x\rangle^{2\mu-1}|\nabla v_\eps(t,x)\overline{v_\eps}(t,x)|e^{-2X_\eps(x)}\drm x\\
&\le2|\mu|\|v_\eps(t)e^{-X_\eps}\|_{L^2}\|\nabla v_\eps(t)\|_{L_{2\mu-1}^2}\\
&\le2|\mu|\|v_0e^{-X_\eps}\|_{L^2}\|v_\eps(t)\|_{\CH_{2\mu-1}^1}
\end{align*}
using the conservation of the modified mass. This gives
\begin{equation*}
\sup_{t\in I}\|v_\eps(t)e^{-X_\eps}\|_{L_\mu^2}^2\le\|v_0e^{-X_\eps}\|_{L_\mu^2}^2+4|I||\mu|\|v_0e^{-X_\eps}\|_{L^2}\sup_{t\in I}\|v_\eps(t)\|_{\CH_{2\mu-1}^1}
\end{equation*}
for any finite interval $I\subset\R$ hence
\begin{align*}
\sup_{t\in I}\|v_\eps(t)\|_{L_{\mu'}^2}^2&\le\|e^{X_\eps}\|_{L_{\mu'-\mu}^\infty}^2\sup_{t\in I}\|v_\eps(t)e^{-X_\eps}\|_{L_\mu^2}^2\\
&\le\|e^{X_\eps}\|_{L_{\mu'-\mu}^\infty}^2\|v_0e^{-X_\eps}\|_{L_\mu^2}+4|I|\mu\|e^{X_\eps}\|_{L_{\mu'-\mu}^\infty}^2\|v_0e^{-X_\eps}\|_{L^2}\sup_{t\in I}\|v_\eps(t)\|_{\CH_{2\mu-1}^1}
\end{align*}
for any $0<\mu'<\mu<\mu_0$. We now prove that $v_\eps(t)$ is bounded in $\CH_{-\mu}^1$ uniformly with respect to $t\in\IR$ and $\eps>0$ using this bound and the conservation of the modified energy. Indeed, we have
\begin{align*}
\int_{\IR^2}|\nabla &v_\eps(t,x)|^2e^{-2X_\eps(x)}\drm x=E_\eps(v_0)+\int_{\R^2}|v_{\eps}(t,x)|^2\big(|\nabla X_\eps(x)|^2+(\varphi*\xi_\eps)(x)-c_\eps\big)e^{-2X_\eps(x)}\dd x\\
&\quad-2\lambda\int_{\IR^2}|v_{\eps}(t,x)|^2 X_\eps(x) e^{-2X_{\eps}(x)} \dd x+\lambda\int_{\IR^2}|v_{\eps}(t,x)|^2 \log |v_{\eps}(t,x)|^2e^{-2X_\eps(x)}\dd x
\end{align*}
and $E_\eps(v_0)\le2|E(v_0)|$ for $\eps$ small enough. For the first two terms, we have
\begin{align*}
\Big|\int_{\R^2}|v_{\eps}(t,x)|^2\big(|\nabla X_\eps(x)|^2+&(\varphi*\xi_\eps)(x)-c_\eps\big)e^{-2X_\eps(x)}\dd x\Big| \\
&\le\||\nabla X_\eps(x)|^2+(\varphi*\xi_\eps)(x)-c_\eps\|_{\CC_{-2\mu}^{-\kappa}}\|v_\eps(t)\overline{v_\eps}(t)\|_{\CB_{1,1,2\mu}^\kappa}\\
&\le2\|\Wick{|\nabla X|^2}\|_{\CC_{-2\mu}^{-\kappa}}\|v_\eps(t)\|_{\CH_\mu^{2\kappa}}^2 \\
&\le2\|\Wick{|\nabla X|^2}\|_{\CC_{-2\mu}^{-\kappa}}\|v_\eps(t)\|_{L_{\frac{2\mu}{1-2\kappa}}^2}^{1-2\kappa}\|v_\eps(t)\|_{\CH_{-\frac{\mu}{2\kappa}}^1}^{2\kappa}&
\end{align*}
and
\begin{align*}
\Big|\int_{\IR^2}|v_{\eps}(t,x)|^2 X_\eps(x) e^{-2X_{\eps}(x)} \dd x\Big|&\le\|X_\eps e^{-2X_\eps}\|_{L_{-\mu}^\infty}\|v_\eps(t)\|_{L_\mu^2}
\end{align*}
for any $\mu,\kappa>0$. For the logarithmic term, we have
\begin{align*}
\Big|\int_{\IR^2}|v_{\eps}(t,x)|^2&\log |v_{\eps}(t,x)|^2e^{-2X_\eps(x)}\dd x\Big|\le\|e^{-2X_\eps}\|_{L_{-\mu}^\infty}\int_{\IR^2}\langle x\rangle^{2\mu}|v_{\eps}(t,x)|^2 \log |v_{\eps}(t,x)|^2\dd x\\
&\le C\|e^{-2X_\eps}\|_{L_{-\mu}^\infty}\|v_\eps(t)\|_{L^2_{\mu_0}}^{\frac{\eta}{\mu_0}+\frac{2\mu}{\mu_0}} \|v_\eps(t)\|_{L^2}^{2-\eta-\frac{\eta}{\mu_0}-\frac{2\mu}{\mu_0}} +C\|e^{-2X_\eps}\|_{L_{-\mu}^\infty}\|v_\eps(t)\|^{2+\eta}_{L^{2+\eta}_{\frac{2\mu}{2+\eta}}}
\end{align*}
using Lemma \ref{logarithmic_lemma} with $\mu<\mu_0$ and $\eta>0$ small enough. Using
\begin{equation*}
\|\nabla v_\eps(t)\|_{L_{-\mu}^2}^2\le\|e^{X_\eps}\|_{L_{-\mu}^\infty}^2 \|\nabla v_\eps(t)e^{-X_\eps}\|_{L^2}^2
\end{equation*}
for any $\mu>0$, we get in the end
\begin{equation*}
\sup_{t\in I}\|\nabla v_\eps(t)\|_{L_{-\mu}^2}^2\le A+B\sup_{t\in I}\|v_\eps(t)\|_{\CH_{2\mu-1}^1}^\beta
\end{equation*}
with $\beta<2$ and $A,B>0$ constants depending on the noise and $v_0$ hence
\begin{equation*}
\sup_{\eps>0}\ \sup_{t\in I}\|v_\eps(t)\|_{\CH_{-\mu}^1}<\infty
\end{equation*}
for $\mu\le\frac13$ to ensure $2\mu-1\le-\mu$. This gives
\begin{equation*}
\sup_{\eps>0}\ \sup_{t\in I}\|v_\eps(t)\|_{L_{\mu_0}^2}<\infty
\end{equation*}
since $\mu_0\le\frac13$ and interpolation yields
\begin{equation*}
\sup_{\eps>0}\ \sup_{t\in I}\|v_\eps(t)\|_{\CH_{(1-\theta)\mu}^\theta}<\infty
\end{equation*}
for any $\theta\in(0,1)$ and $\mu<\mu_0$. For the $\CH^2$ bound, we use again Kato's trick with $w_\eps=\partial_tv_\eps$. Using the equation at the level of $\partial_tu_\eps$, we have
\begin{equation*}
\|e^{-X_\eps}w_\eps(t)\|_{L^2}^2\le e^{4|\lambda t|}\|e^{-X_\eps}w_\eps(0)\|_{L^2}^2
\end{equation*}
while the equation on $v_\eps$ gives
\begin{equation*}
w_\eps(0)=\Delta v_0-2\nabla X_\eps\cdot\nabla v_0+(|\nabla X_\eps|^2+\varphi*\xi_\eps-c_\eps)v_0+\lambda v_0\log|v_0|^2.
\end{equation*}
As $\eps$ goes to $0$, $w_\eps(0)$ converges to a distribution in $\CH_\mu^{-\kappa}$ for any $\mu<\mu_0$ and $\kappa>0$ since $v_0\in\CH_{\mu_0}^2$. Thus $\|e^{-X_\eps}w_\eps(0)\|_{L^2}$ diverges as $\eps$ goes to $0$ due to local irregularity. Note that on the torus, it is enough to use Gagliargo-Nirenberg inequality since $\nabla X_\eps$ and $\Wick{|\nabla X_\eps|^2}$ diverges as $\log(\eps)$ in $L^4(\IT^2)$, see \cite{debussche2018}. Following \cite{debussche2019}, we use instead
\begin{equation*}
\| \nabla X_\eps\|_{L^p_{-\mu}(\R^2)}^2 + \|\Wick{|\nabla X_\eps|^2}\|_{L^p_{-\mu}(\R^2)}\le C|\log \eps|
\end{equation*}
for $\mu\in(0,1)$ and $p>\frac{2}{\mu}$. For $\frac{1}{2}=\frac{1}{p}+\frac{1}{q}$, Hölder inequality gives
\begin{align*}
\|\nabla X_\eps\cdot\nabla v_0\|_{L_\mu^2}+\|\Wick{|\nabla X_\eps|^2}v_0\|_{L_\mu^2}&\le C\|\nabla X_\eps\|_{L_{-\mu}^p}\|\nabla v_0\|_{L_{2\mu}^q}+\|\Wick{|\nabla X_\eps|^2}\|_{L_{-\mu}^p}\|v_0\|_{L_{2\mu}^q}\\
&\le C'|\log\eps|\big(\|\nabla v_0\|_{\CH_{2\mu}^{1-\frac{2}{q}}}+\|v_0\|_{\CH_{2\mu}^{1-\frac{2}{q}}}\big)\\
&\le C'|\log\eps|\|v_0\|_{\CH_{\mu_0}^2}
\end{align*}
for $2\mu\le\mu_0$. Using Lemma \ref{logarithmic_lemma}, we get
\begin{equation*}
\|w_\eps(0)\|_{L_\mu^2}\le C|\log(\eps)|\|v_0\|_{\CH_{\mu_0}^2}
\end{equation*}
thus
\begin{align*}
\|w_\eps(t)\|_{L_{-\mu}^2}&\le\|e^{X_\eps}\|_{L_{-\mu}^\infty}\|e^{-X_\eps}w_\eps(t)\|_{L^2}^2\\
&\le C|\log(\eps)|e^{4|\lambda t|}\|e^{X_\eps}\|_{L_{-\mu}^\infty}\|e^{-X_\eps}\|_{L_{-\mu}^\infty}\|v_0\|_{\CH_{\mu_0}^2}
\end{align*}
for any $\mu\in(0,1)$. Using the equation on $v$ and similar bounds, this gives
\begin{equation*}
\|v_\eps(t)\|_{\CH_{-\mu}^2}\le C|\log\eps|^\beta
\end{equation*}
for some $\beta>0$ and random positive constant $C>0$ independant of $\eps\in(0,1)$. As announced before, this $\CH^2_{-\mu}$ bound is not uniform with respect to $\eps$ so we can not conclude directly by compactness arguments. Instead, as performed in \cite{debussche2018}, we are going to show that the sequence $(v_{\eps})_{\eps}$ is a Cauchy sequence in $L^2_{\mu}$ with a polynomial rate in $\eps$, hence we will recover a limit in $\CH^{2\theta}_{(1-\theta)\mu}$ by interpolation between $L^2_{\mu}$ and $H^2_{-\mu}$, taking advantage of the fact that the divergence in $\CH^2_{-\mu}$ is only logarithmic in $\eps$ and can therefore be absorb by a polynomial decrease in $\eps$. Let's denote $\eps_2>\eps_1>0$ and $r=v_{\eps_1}-v_{\eps_2}$, which satisfies
\begin{align*} 
i \partial_t r &= \Delta r -2 \nabla r \cdot \nabla X_{\eps_1}+ r(\Wick{|\nabla X_{\eps_1}|^2}-\varphi*\xi_{\eps_1}) -2\lambda X_{\eps_1} r +2 \nabla v_{\eps_2} \cdot \nabla (X_{\eps_1} -X_{\eps_2}) \\
&\quad+  v_{\eps_2} (\Wick{|\nabla X_{\eps_1}|^2}-\Wick{|\nabla X_{\eps_2}|^2}-\varphi*\xi_{\eps_1}+\varphi*\xi_{\eps_2})+2\lambda v_{\eps_2} ( X_{\eps_1} - X_{\eps_2}) \\
&\quad+\lambda (v_{\eps_1} \log |v_{\eps_1}|^2 - v_{\eps_2} \log |v_{\eps_2}|^2 )
\end{align*}
with $r(0)=0$. Multiplying the equation by $\overline{r}$, integrating over $\R^2$ and taking the imaginary part, we get
\begin{align*}
&\frac12 \frac{\dd}{\dd t} \Big( \int_{\R^2} |r(t,x)|^2 e^{-2 X_{\eps_1}(x)} \dd x \Big)= \Im \int_{\R^2} 2 \nabla v_{\eps_2}(t,x) \cdot \nabla (X_{\eps_1} -X_{\eps_2})(x)\overline{r}(t,x) e^{-2 X_{\eps_1}(x)}\dd x \\
&+\Im\int_{\R^2}v_{\eps_2}(t,x)\big(\Wick{|\nabla X_{\eps_1}|^2}-\Wick{|\nabla X_{\eps_2}|^2}-\varphi*\xi_{\eps_1}+\varphi*\xi_{\eps_2}\big)(x)\overline{r}(t,x) e^{-2 X_{\eps_1}(x)}\dd x  \\
& + \lambda \Im \int_{\R^2} \big( 2 v_{\eps_2}(t,x) ( X_{\eps_1} - X_{\eps_2})(x) +(v_{\eps_1} \log |v_{\eps_1}|^2 - v_{\eps_2} \log |v_{\eps_2}|^2 )(t,x) \big)\overline{r}(t,x) e^{-2 X_{\eps_1}(x)}\dd x.
\end{align*}
Using Lemma \ref{estimate_diff_Y_R2} alongside the previous bounds, this gives
\begin{align*}
\Big| \int_{\R^2} \nabla v_{\eps_2}(t,x) \cdot \nabla (X_{\eps_1}-X_{\eps_2})(x) &\overline{r}(t,x) e^{-2X_{\eps_1}(x)} \dd x \Big|\\
& \leq\| \nabla X_{\eps_1}- \nabla X_{\eps_2} \|_{\mathcal{C}^{-\frac12+\kappa}_{-\mu'}} \| \nabla v_{\eps_2} \overline{r} e^{-2X_{\eps_1}} \|_{\mathcal{B}^{\frac12-\kappa}_{1,1,\mu'}} \\
& \leq C \eps_2^{\kappa} \|\nabla v_{\eps_2} \|_{\CH^{\frac12-\frac{\kappa}{2}}_{\frac{\mu'}{2}}} \| r \|_{\CH^{\frac12}{\mu'}} \| e^{-2X_{\eps_1}} \|_{\mathcal{C}^{\frac12}_{-\frac{\mu'}{2}}} \\
& \leq C \eps_2^{\kappa} \| v_{\eps_2} \|_{\CH^{\frac{3}{2}}_{\frac{\mu'}{2}}} \left( \| v_{\eps_1} \|_{\mathcal{C}_T \CH^{\frac12}_{\mu'}} + \| v_{\eps_2} \|_{\mathcal{C}_T \CH^{\frac12}_{\mu'}}  \right) \\
& \leq C\eps_2^{\kappa}|\log \eps_2 |^{\frac{3}{2}} e^{3|\lambda|t} \|v_0 \|_{\CH^2_{\mu_0}}^{\alpha}
\end{align*}
and
\begin{align*}
\Big| \int_{\R^2} v_{\eps_2}(t,x)&(\Wick{|\nabla X_{\eps_1}|^2}-\Wick{|\nabla X_{\eps_2}|^2}-\varphi*\xi_{\eps_1}+\varphi*\xi_{\eps_2})(x)\overline{r}(t,x)e^{-2X_{\eps_1}(x)}\dd x  \Big|\\
& \leq\|\Wick{|\nabla X_{\eps_1}|^2}-\Wick{|\nabla X_{\eps_2}|^2}-\varphi*\xi_{\eps_1}+\varphi*\xi_{\eps_2} \|_{\mathcal{C}^{-\frac12+\kappa}_{-\mu'}} \| \nabla v_{\eps_2} \overline{r} e^{-2X_{\eps_1}} \|_{\mathcal{B}^{\frac12-\kappa}_{1,1,\mu'}} \\
& \leq C \eps_2^{\kappa} \| v_{\eps_2} \|_{\CH^{\frac{1}{2}}_{\frac{\mu'}{2}}} \| r \|_{\CH^{\frac12}_{\mu'}} \\
&\leq C \eps_2^{\kappa} \|v_0 \|_{\CH^1_{\mu_0}}^{\alpha'}
\end{align*}
as soon as $\kappa \in \left(0,\frac12 \right)$, $\mu'<\frac{\mu_0}{2}$ and constants $\alpha,\alpha'>0$. The linear term is handle the same way and we make use of the inequality \eqref{log_equation_cazenave} for the logarithmic part to get
\begin{align*} 
\Big| \int_{\R^2} \Im \left( (v_{\eps_1} \log |v_{\eps_1}|^2 - v_{\eps_2} \log |v_{\eps_2}|^2 )(\overline{v_{\eps_1}} -  \overline{v_{\eps_2}}) \right)&(t,x)e^{-2X_{\eps_1}(x)}\dd x  \Big| \\
&\leq 4 \int_{\R^2}|r(t,x)|^2e^{-2X_{\eps_1}(x)}\dd x.
\end{align*}
Gathering these bounds, Gronwall lemma thus yields
\begin{equation*}
\| v_{\eps_1}-v_{\eps_2} \|_{\mathcal{C}(I; L^2_{-\mu})} \leq C \eps_2^{\kappa/2}|\log \eps_2|^{\frac{3}{2}} e^{4|\lambda||I|} \|v_0\|_{\CH^2_{\mu_0}}^\beta
\end{equation*}
for $I \subset \R$ and some numeric constant $\beta>0$. By interpolation we then infer that 
\begin{equation*}
\|v_{\eps_1}-v_{\eps_2} \|_{\mathcal{C}_T \CH^{\gamma}_{\mu}} \leq 2 \|v_{\eps_1}-v_{\eps_2} \|_{\mathcal{C}_T L^2_{-\mu_2}}^{1-\frac{\gamma}{\gamma'}} \|v_{\eps_2}\|_{\CH^{\gamma'}_{\mu_1}}^{\frac{\gamma}{\gamma'}}
\end{equation*}
with $\gamma'\in(1,2)$, $\gamma<\gamma'$, $\mu_1<(1-\gamma/2)\mu_0$ and $ \mu=-\mu_2 +\frac{\gamma}{\gamma'}(\mu_1+\mu_2)$
which leads to 
\begin{equation*}
\|v_{\eps_1}-v_{\eps_2} \|_{\mathcal{C}(I,\CH^{\gamma}_{\mu})} \leq C \eps_2^{\frac{\kappa}{2} \left( 1-\frac{\gamma}{\gamma'} \right)} |\log \eps_2|^{\frac{3}{4}\left( 1-\frac{\gamma}{\gamma'} \right)} \|v_0\|_{H^2_{\mu_0}}^{a(\gamma,\gamma')}
\end{equation*}
where $a(\gamma,\gamma')>0$, so $(v_{\eps})_{\eps}$ is well a Cauchy sequence in $\mathcal{C}\big(I,\CH^{\gamma}_{\mu}(\R^2)\big)$ by comparative growth, and so there exists a limit function $v\in \mathcal{C}\big(I,\CH^{\gamma}_{\mu}(\R^2)\big)$ solution of the limit equation with initial data $v_0$.

\medskip

It now remains to show the pathwise uniqueness. Let $v_1$ and $v_2$ be two solutions of \eqref{SLNLSmodif} with paths in $\CC\big(I,\CH^{\gamma}_{\mu}(\R^2)\big)$ with same initial data $v_0 \in \CH^2_{\mu_0}(\R^2)$. We set $r=v_1-v_2$, which satisfies the equation
\begin{equation*}
i \partial_t r = \Delta r -2 \nabla r \cdot \nabla X + r(\Wick{|\nabla X|^2}-\varphi*\xi) -2\lambda Xr +\lambda (v_1 \log |v_1|^2 - v_2 \log |v_2|^2 ),
\end{equation*}
leading to the standard estimate
\begin{align*}
\frac{1}{2} \frac{\dd}{\dd t}  \int_{\R^2} |r(t,x)|^2 e^{-2 X(x)}\dd x &= \lambda \Im \int_{\R^2} \left(v_1 \log |v_1|^2 - v_2 \log |v_2|^2  \right)(t,x)\overline{r}(t,x) e^{-2 X(x)}\dd x \\
&\leq 4 |\lambda| \int_{\R^2} |r(t,x)|^2 e^{-2 X(x)}\dd x
\end{align*}
using equation \eqref{log_equation_cazenave}. By Gronwall lemma, we get that
\begin{equation*}
\int_{\R^2} |r(t,x)|^2 e^{-2 X(x)}\dd x \leq e^{8|\lambda|t} \int_{\R^2} |r(0,x)|^2 e^{-2 X(x)}\dd x
\end{equation*}
hence $r=0$ since $r(0)=0$ and this completes the proof.
\end{proof}

\subsection*{Acknowledgements}
Q.C. is supported by the Labex CEMPI (ANR-11-LABX-0007-01). A.M. is supported by a Simons Collaboration Grant on Wave Turbulence.

\bibliographystyle{siam} 
\bibliography{biblio.bib}

\vspace{2cm}

\noindent \textcolor{gray}{$\bullet$} Q. Chauleur -- INRIA Lille, Univ Lille \& Laboratoire Paul Painlevé, CNRS UMR 8524 Lille, Cité Scientifique, 59655 Villeneuve-d'Ascq, France.\\
{\it E-mail}: quentin.chauleur@math.cnrs.fr

\noindent \textcolor{gray}{$\bullet$} A. Mouzard -- ENS de Lyon, CNRS, Laboratoire de Physique, F-69342 Lyon, France.\\
{\it E-mail}: antoine.mouzard@math.cnrs.fr

\end{document}